\documentclass[a4paper,11pt]{article}
\usepackage{mathtext}         %

\usepackage{latexsym}
\usepackage{amsmath,amssymb,amsthm,amsopn,amscd}
\usepackage[dvips]{graphicx}
\usepackage{array,fullpage,wrapfig}

\usepackage[matrix,arrow,curve]{xy}
\usepackage[active]{srcltx}

\usepackage{euscript}

\usepackage{ulem}
\usepackage[usenames,dvipsnames]{color}
\def\cred{\color{red}} \def\cbl{\color{blue}}  

\newcommand\g{{\mathbf{g}}}
\newcommand\h{{\mathbf{h}}}
\newcommand{\p}{{\mathbf{p}}}
\newcommand{\n}{{\mathbf{n}}}

\newcommand\gl{{\bf gl}}

\newcommand\Wng{W_n\ltimes \g\otimes\calO_n}
\newcommand\gtg{\omega(\g)}   
\newcommand\W{\EuScript{W}}

\newcommand\Wnm{{ W_n\ltimes W_m\otimes\calO_n}}

\newcommand{\Z}{{\mathbb{Z}}}
\newcommand\C{{\mathbb C }}
\newcommand\R{{\mathbb R }}
\newcommand\?{{\rm M}}

\newcommand{\quis}{quis}

\def\b{{\mathbf p}_{m n}} 

\newcommand\bm{{{\mathbf p}_{m n}^{-}}}
\newcommand\nm{{{\mathbf n}^{-}_{m n}}}

\newcommand\np{{{\mathbf n}^{+}_{m n}}}

\newcommand\VU{{\hbox to 11pt {V\hss U}}}

\newcommand\calO{{\mathcal O}}

\newcommand\kk{\Bbbk}
\def\?{{\rm M}}
\newcommand\F{{\mathcal F}}

\newcommand\dn{\overline{n}}

\newcommand{\dx}[1]{\frac{\partial}{\partial x_{#1}}}

\newcommand\rH{{\rm H}}

\newcommand{\ldot}{{\:\raisebox{1.5pt}{\selectfont\text{\circle*{1.5}}}}}
\newcommand{\udot}{{\:\raisebox{4pt}{\selectfont\text{\circle*{1.5}}}}}
\newcommand\pt{\udot}

\newcommand\GL{\mathbf{GL}}

\newcommand\OP{\mathop\oplus\limits}
\newcommand\OT{\mathop\otimes\limits}

\newcommand\ttt{\text{-}}

\DeclareMathOperator{\Hom}{Hom}

\let\ge\geqslant
\let\leq\leqslant
\let\geq\geqslant
\let\cong\simeq

\DeclareMathOperator*{\ooplus}{{\oplus}}

\newcommand{\rf}[1]{(\ref{#1})}
\newcommand{\picref}[1]{\ref{#1}}

\newtheorem{theorem}[equation]{Theorem}
\newtheorem*{theorem*}{Theorem}
\newtheorem{proposition}[equation]{Proposition}
\newtheorem*{proposition*}{Proposition}
\newtheorem{statement}[equation]{Statement}
\newtheorem{lemma}[equation]{Lemma}
\newtheorem{corollary}[equation]{Corollary}
\newtheorem*{corollary*}{Corollary}

\theoremstyle{definition}

\newtheorem{definition}[equation]{Definition}

\theoremstyle{remark}
\newtheorem{remark}[equation]{Remark}
\begin{document}

\title{
Characteristic classes of flags of foliations \\
and \\
Lie algebra cohomology
}

\author{Anton Khoroshkin
\thanks{
The research is partially supported by RFBR grants 13-02-00478, 13-01-12401,  
by "The National Research University--Higher School of Economics" Academic Fund Program in 2013-2014, 
research grant 14-01-0124, by Dynasty foundation and Simons-IUM fellowship.
}
}
\date{}

\maketitle

\begin{abstract}
We prove the conjecture by Feigin, Fuchs and Gelfand 
{
describing
} 
the Lie algebra cohomology of formal vector fields on 
{
an
}
 $n$-dimensional space
with coefficients in symmetric powers of the coadjoint representation.
We also compute the cohomology of the Lie algebra of formal vector fields that preserve a given flag at the origin.
The latter encodes characteristic classes of flags of foliations and 
{
was
} 
used in the 
formulation of the
local Riemann-Roch Theorem by Feigin and Tsygan.

Feigin, Fuchs and Gelfand described the first symmetric power and to do this they had to make use of a fearsomely complicated computation in invariant theory. 
By the application of degeneration theorems of appropriate Hochschild-Serre spectral sequences we avoid the need to use the methods of FFG, 
and moreover we are able to describe all the symmetric powers at once.
\end{abstract}

\tableofcontents

\setcounter{section}{-1}

\section{Introduction}

\subsection{Main results}
The main result of this paper is 
a computation of
the homology 
of 
the Lie algebra $W_n$ of formal vector fields on 
an
$n$--dimensional space.
$$
W_n:= \{\sum_{i=1}^{n} f_i \frac{\partial}{\partial x_i} | f_i\in \kk[[x_1,\ldots,x_n]] \}
$$
In the early 70's B.\,Feigin, D.\,Fuchs and I.\,Gelfand stated  a conjectural description of its cohomology  
with coefficients in symmetric powers of the coadjoint representation. They 
described
applications of this conjecture 
to formal geometry, Gelfand-Fuchs cohomology and 
the
theory of foliations. 
This
problem was formulated 
at
the famous Gelfand seminar at Moscow State University.
In~\cite{GFF} the same authors 
confirmed their conjecture for the 
particular case of the first symmetric power of the coadjoint representation.
Later on, in 1989, this cohomological conjecture was 
offered as a proposition\footnote{without a proof}
 by B.\,Feigin and B.\,Tsygan 
in \cite{Feigin::Tsygan} and was used in the proof of the local Riemann-Roch theorem. 
In 2003, V.\,Dotsenko (\cite{Dotsenko}) directly computed the cohomology for the linear dual problem in the case $n=1$.
Namely, he computed the homology of the Lie algebra of polynomial vector fields on the line with coefficients in 
symmetric powers of the
adjoint representation. 

In this article a uniform method will be presented, 
which does not involve lengthy computations with invariants and by means of which the cohomologies 
may be described for all cases of positive $n$ and all symmetric powers at once. 
The main results can be summarised in the following two statements, proved in Sections 2.3 and 2.2 respectively.
\begin{theorem*}[Theorem~\ref{thm::sym_coef}, Corollary~\ref{thm::Wn:SWn:relative}]
 For all $k\ge1$ the cohomology of the Lie algebra $W_n$ relative to the subalgebra $\gl_n$ of linear vector fields
with coefficients in
 the 
$k$-th symmetric power of 
 the 
coadjoint representation vanishes everywhere except 
for
degree $2n$ 
and 
the $2n$th 
cohomology 
coincides 
with the space of $\gl_n$-invariants in the $(n+k)$-th symmetric power of adjoint representation $\gl_n$:
\begin{equation}
\label{eq::H:Wn::intro}
\rH^{2n}(W_n,\gl_n;S^{k}W_n^{*}) = [S^{n+k}\gl_n]^{\gl_n}. 
\end{equation}
The absolute cohomology has the following description:
$$\rH^{i}(W_n;S^{k}W_n^{*}) =
\left\{ \begin{array}{l}
         {[S^{n+k}\gl_n]^{\gl_n}\otimes [\Lambda^{i-2n}(\gl_n)]^{\gl_n}, {\mbox{ if }} 2n\leq i\leq n^2+2n,}\\
         {0,  {\mbox{ otherwise. }}}
       \end{array}
\right.
$$
\end{theorem*}
Recall that the algebra of $\gl_n$--invariants in the symmetric algebra $S^{\udot}(\gl_{n})$
is isomorphic to the symmetric algebra
with $n$ generators of degrees $1,2,\ldots, n$.
Also, the algebra of $\gl_n$--invariants in the exterior algebra 
$\Lambda^{\udot}(\gl_n)$ is isomorphic to the exterior algebra 
with $n$ odd generators of degrees $1,3,5,\ldots,2n-1$.
This is explained for example in
\cite{H_Weyl,Fuks::Lie::cohomology}.
In particular, we can identify the 
$2n$th
cohomology $\rH^{2n}(W_n;S^{k}W_n^{*})$ 
of order $2n$
with the subspace of polynomials 
of degree $n+k$ in the polynomial algebra $\kk[x_1,x_2,\ldots,x_n]$ subject to the 
condition that
the generator $x_i$ has degree $i$.
Formulas for the generating cocycles of~\rf{eq::H:Wn::intro} are given in Section~\S\ref{sec::cocycles}.

At the same time we prove another homological conjecture initially formulated in formal geometry 
and in 
foliation theory (see e.g.~\cite{Feigin::flag::foliations}):
we compute the cohomology of the Lie algebra of formal vector fields that preserve a given flag of foliations. 
More precisely, let us fix 
a 
collection of natural numbers $\dn=(n_0,\dots,n_k)$ with $|\dn|=n_0+\ldots+n_k$,
and
fix a sequence of trivial embedded foliations $\{\F_1,\F_2,\dots,\F_k\}$ in $\mathbb{R}^{|\dn|}$
 such that the foliation $\F_{i+1}$ has codimension $n_i$ in $\F_i$ ($n_0$ is 
the 
codimension of $\F_1$  in $\mathbb{R}^{|n|}$.
The Lie algebra $W(n_0,\ldots,n_k)$ consists of formal vector fields that infinitesimally 
preserve all foliations near the origin (see section \ref{sec::Wn::def} for details). That is
$$
W(n_0,\ldots,n_k) := 
\left\{\sum_{i=1}^{|\dn|} f_i\frac{\partial}{\partial x_i} 
\left|
\begin{array}{c}
f_1,\ldots, f_{n_0}\in \R[[x_1,\ldots,x_{|\dn|}]], \\
f_{n_0+1},\ldots, f_{n_0+n_1} \in \R[[x_{n_0+1},\ldots,x_{|\dn|}]], \\
f_{n_0+n_1+1},\ldots, f_{n_0+n_1+n_2} \in \R[[x_{n_0+n_1+1},\ldots,x_{|\dn|}]], \\
\ldots
\end{array}
\right.
\right\}
$$
Consider the maximal reductive subalgebra in the space of linear vector fields that preserve the flag mentioned above.
This subalgebra is isomorphic to the direct sum of matrix algebras $\gl_{n_0}\oplus\ldots\oplus\gl_{n_k}$. 
Let $I_{n_0,\ldots,n_k}$ be the ideal in the symmetric algebra $S^{\pt}(\gl_{n_0}\oplus\ldots\oplus\gl_{n_k})$
generated by the sum $\OP_{r=0}^{k} S^{n_0+\ldots+n_r+1}(\gl_{n_0}\OP\ldots\OP\gl_{n_r})$.
We state the 
description of
 the relative cohomology in terms of the quotient algebra:
\begin{theorem*}[Cor.\ref{thm::relative_har_flag_sloj} below]
The cohomology of the Lie algebra $W(n_0,\ldots,n_k)$ relative to subalgebra $\gl_{n_0}\oplus\ldots\oplus\gl_{n_k}$
with trivial coefficients
are different from zero only in even degrees and coincides with the algebra of $\gl_{n_0}\oplus\ldots\oplus\gl_{n_k}$
invariants in the factor-algebra 
of the symmetric algebra $S^{\pt}(\gl_{n_0}\oplus\ldots\oplus\gl_{n_k})$ by the ideal $I_{n_0,\ldots,n_k}$. 
In particular, for all $i$ we have:
$$\rH^{i}(W(n_0,\ldots,n_k),\gl_{n_0}\oplus\ldots\oplus\gl_{n_k};\kk) =
\left\{
\begin{array}{l}
{ 0, \text{ if } i = 2k+1,} \\
{\left[\frac{S^{k}(\gl_{n_0}\oplus\ldots\oplus\gl_{n_k})}{I_{n_0,\ldots,n_k}}\right]^{\gl_{n_0}\oplus\ldots\oplus\gl_{n_k}}, \text{ if } i =2k.}
\end{array}
\right.
$$
\end{theorem*}
Moreover, the algebra of $\gl$ invariants has canonical description as a ring of symmetric functions:
$$\left[S^{\pt}(\gl_{n_0}\oplus\ldots\oplus\gl_{n_k})\right]^{\gl_{n_0}\oplus\ldots\oplus\gl_{n_k}} = 
\kk\left[\Psi_{0 1},\dots,\Psi_{0 {n_0}};\dots;\Psi_{k 1},\dots,\Psi_{k {n_k}}\right].$$
Namely, the aforementioned ring is the free commutative algebra generated by the set of generators 
$\{\Psi_{i j} \}$, index $i$ ranges from $0$ to $k$ and $j$ ranges from $1$ to $n_i$.
The degree of the generator $\Psi_{i j}$ coincides with the second index $j$.
The intersection of the ideal $I_{n_0,\ldots,n_k}$ with the subalgebra of invariants is generated by monomials 
$(\Psi_{0 1}^{\alpha_{0 1}}\dots\Psi_{0 n_0}^{\alpha_{0 n_0}})
\dots (\Psi_{r 1}^{\alpha_{r 1}}\dots\Psi_{r n_r}^{\alpha_{r n_r}})$ such that 
 $(\sum_{i=0}^{r}\sum_{j=1}^{n_i}\alpha_{i j}) > n_0+\dots +n_r,$
where the index $r$ ranges from $0$ to $k$ as above.

The answer for the absolute cohomology of the Lie algebra $W(n_0,\ldots,n_k)$
may be complicated for particular choice of integer numbers $(n_0,\ldots,n_k)$,
however there exists a straitforward procedure on how one should complute this absolute cohomology.
We refer these discussions to Section~\ref{sec::W_flag} and
in particular to Theorem~\ref{thm::W_flag_absolute}.
The particular case $\dn=(1,1,\ldots,1)$ where all dimensions are computed in terms of Catalan numbers
is considered in Section~\ref{sec::full_flag_formulas}.

\subsection{Motivations: formal geometry}

In the early 1970s I.\,M.\,Gelfand presented the description of characteristic classes of a manifold of dimension $n$
using the Lie algebra cohomology of the infinite-dimensional Lie algebra of formal vector fields $W_n$.
(See e.g.~\cite{GK} written after presentations at the Gelfands famous seminar, and an overview after ICM in Nice~\cite{G_ICM};
more details can be found in \cite{Fuks::Lie::cohomology}{ch.3},\cite{GK},\cite{BR},\cite{Bott_har_cl}.)

The core construction is called \emph{formal geometry} or \emph{Gelfand-Fuchs cohomology} and looks as follows. 

To 
 any complex smooth manifold $M$ of dimension $n$ we assign the infinite-dimensional manifold $M^{coor}$
of formal coordinates on $M$.
Formally,
a
point of $M^{coor}$ is a pair:
a point $p\in M$ together with a system of formal coordinates in 
a neighbourhood 
of $p$.
In other words, the manifold $M^{coor}$ is the space of $\infty$-jets of submersions $M\rightarrow \C^{n}$
with $0$ as 
target.

The tangent space at each point of $M^{coor}$ is isomorphic to the Lie algebra 
of formal vector fields $W_{n}$. 
The union of these isomorphisms defines a map from the Chevalley--Eilenberg complex of the Lie algebra $W_n$ 
to the De Rham complex of $M^{coor}$.
However, there is a canonical fibration $M^{coor}\rightarrow M$ whose fibers are homotopy equivalent to the Lie group 
$\GL_n(\C)$ and consequently to its compact subgroup $U(n)$ of unitary matrices.
Thus, we get the following morphism of relative complexes:
$$ch: C^{\udot}(W_n,\gl_n;\R) \rightarrow \Omega_{DR}^{\udot}(\frac{M^{coor}}{U(n)}).$$
and the corresponding characteristic map of cohomology rings:
$$
ch: \rH^{\udot}(W_n,\gl_n;\R) \rightarrow \rH^{\udot}_{DR}(\frac{M^{coor}}{U(n)}) = \rH^{\udot}_{DR}(M).
$$
This map 
describes 
characteristic classes of 
the 
tangent bundle 
using the Lie algebra cohomology of infinite-dimensional Lie algebras.
The same construction works for real manifolds.
The main difference is that 
one has to replace $GL_n(\C)$ by $GL_n(\R)$ and the compact subgroup $U(n)\subset GL_n(\C)$ by a compact Lie group $O(n)\subset GL_n(\R)$ of orthogonal matrices
and get a characteristic map:
$$
ch: \rH^{\udot}(W_n,{\mathsf{o}}_n;\R) \rightarrow \rH^{\udot}_{DR}(\frac{M^{coor}}{O(n)}) = \rH^{\udot}_{DR}(M).
$$

 The procedure suggested above depends on the local data.
 One can follow 
the same approach while working with
manifolds 
having
additional geometric data.
For example, 
one considers the manifold $M_{\F}^{coor}$ of
system of $n$ formal coordinates 
in the directions transversal to the leaves of a given foliation $\F$ 
of
a manifold $M$ of codimension $n$.
The tangent space to the infinite-dimensional manifold $M_{\F}^{coor}$
will be isomorphic to the Lie algebra $W_n$. 
This provides 
a 
description of the characteristic classes of foliations generalizing the 
Godbillon-Vey characteristic class.
A similar
construction may be given if one is interested in the characteristic classes of 
$G$-bundles over $M$ 
for
a compact group $G$ (see e.g.~\cite{Khor::Wn_g}).
The cohomological results we obtain below are applied for 
the following constructions of the same kind:
\begin{enumerate}
 \item 
To each foliation with 
prescribed codimension $d$ 
of
an $n$-dimensional manifold
we assign 
an
infinite-dimensional manifold of formal coordinates on the underlying manifold such that 
the first $d$ coordinates should be 
coordinates on a leaf.

In particular, the above situation is interesting even in the case of 
a
fibration $f: S\mapsto M$
where the fiber over each point $p\in M$ is a compact manifold of 
some
given dimension.
(See e.g.\cite{Feigin::Tsygan} for details and \cite{Felder,FFSh,FLSh} for applications to the local Riemann-Roch theorem).

 \item
Consider a flag of foliations.
The assigned infinite-dimensional manifold will consists 
of formal coordinates preserving leaves.
That is, the first $n_1$ coordinates will form coordinates transversal to leaves of the foliation with smallest codimension,
first $n_1+n_2$ coordinates will form coordinates transversal to leaves of the second foliation and so on.
(See \cite{Feigin::flag::foliations} for detailed construction.)

\end{enumerate}
The corresponding infinite-dimensional Lie algebra in the first example is known to coincide
with $W({d,n-d})$. In the second example one has to compute the cohomology of 
the Lie algebra $W(n_1,\ldots,n_k)$, where $(n_1,\ldots,n_k)$ is the collection of codimensions.
Thus, the cohomological computations we made 
lead to 
a
description of characteristic classes 
of
flags of foliations. 

Let us mention a particular application to the theory of foliations.
Let $\Gamma$ be a discrete subgroup of the Lie group $SL_{n+1}(\R)$.
Let $\mathsf{P}$ be a subgroup of $SL_{n+1}(\R)$ that fixes a given line in $\R^{n+1}$.
Then $\mathsf{P}$ acts on $SL_{n+1}(\R)/\Gamma$ by left translations and the orbits of this action defines 
a foliation $\F_{\mathsf{P}}$ of codimension $n$.
The following Corollary is proved in Section~\ref{sec::obstruction}:
\begin{corollary*}[Proposition~\ref{cor::SL::foliation}]
The foliation $\F_{\mathsf{P}}$ on $SL_{n+1}(\R)/\Gamma$ can not be a subfoliation.
That is, there is a cohomological obstruction to the existence of a flag of foliations $\F_1\supset \F_2$ on $SL_{n+1}(\R)/\Gamma$ 
such that the foliation $\F_2$ is concordant to $\F_{\mathsf{P}}$.
\end{corollary*}

\subsection{Outline of the paper}
The paper is organized as follows:

Section~\ref{sec::notation} contains some standard definitions (Subsection~\ref{sec::Lie::H::def}),
a
description of all infinite-dimensional Lie algebras whose cohomology we compute (Subsection~\ref{sec::Wn::def})
and 
a
description of the acyclic Weyl superalgebra (Subsection~\ref{sec::Weyl::super}).

Section~\ref{sec::chains::Weyl::Wn} contains statements of all the main results of this paper.
We start from Subsection~\ref{sec::Wn::known} where we recall a classical 
result by Gelfand and Fuchs and its generalization due to the author.
In Subsection~\ref{sec::main::res} we formulate the main Theorems~\ref{thm::homol::WLmn} and~\ref{thm::HWm|nnng}
on relative cohomology of Lie algebras $WL(m|n)$ and $W(n_1,\ldots,n_k)$. In Subsection~\ref{sec::W_flag}
we describe absolute cohomologies 
in terms of 
relative ones.
A particular cohomological obstruction to the existence of 
a
flag of foliations is given in Subsection~\ref{sec::obstruction}.

Section~\ref{sec::main::methods} contains 
a
proof of the main Theorem~\ref{thm::homol::WLmn}.
We state several homological results on 
the
relative homology of 
a
parabolic subalgebra in Subsection~\ref{sec::parabolic::gl},
however, the proofs of these results are postponed to Appendix~\ref{sec::parabolic}.
In Subsection~\ref{sec::degen::Hoch::Serre} we use these results in order to describe the Hochschild-Serre spectral sequences 
associated with
embeddings of a parabolic subalgebra
into appropriate infinite-dimensional Lie 
subalgebras of the Lie algebra $W_n$
of formal vector fields.
In Subsection~\ref{sec::proof::HWLmn} we show how  
degenerations of 
the
aforementioned Hochschild-Serre spectral sequences
imply 
the
main Theorem~\ref{thm::homol::WLmn}.

We illustrate the combinatorics of our main result in two particular cases in Section~\ref{sec::examples}.
Namely, we give 
a
description of the representing cocycles in the absolute case $W(1,\ldots,1)$ 
in Subsection~\ref{sec::full_flag_formulas}.
and  for the Lie algebra cohomology $\rH^{\udot}(W_1;S^{m}W_1^{*})$
in Subsection~\ref{sec::cocycles::W1}. 
Description
of representing cocycles for arbitrary $n$ requires more definitions,
and, therefore, this
is presented in Appendix~\ref{sec::cocycles}. 

Appendix~\ref{sec::chains::gl::inv} deals with a careful description of the relative chains
and cocycles for 
the
infinite-dimensional Lie algebras under consideration. 

Apendix~\ref{sec::parabolic} contains general results on cohomology of parabolic Lie subalgebras 
based on the existence of BGG resolution. The particular case of upper block-triangular matrices leads 
to Lemmas~\ref{lm::H_b::coeff} and~\ref{lm::H_gl::to::H_b} which was used in the proof of 
the main Theorem~\ref{thm::homol::WLmn}.

\subsection{Acknowledgments}
I am grateful to V.~Dotsenko, B.\,Shoikhet, D.\,B.\,Fuchs  
for useful discussions.
Special thanks are addressed to Boris Shoikhet for his explanations of the importance of this problem and
to my advisor Boris Feigin for many useful and stimulating discussions.
My conversations with him about this problem took place over several years.

\section{Notations and recollections}
\label{sec::notation}

\subsection{Lie algebras of formal vector fields and their subalgebras}
\label{sec::Wn::def}
The main infinite-dimensional Lie algebra in this article is 
the
Lie algebra of formal vector fields on the $n$-dimensional plane.
We use the standard notation $W_n$ 
for this Lie algebra
following textbook~\cite{Fuks::Lie::cohomology}.

Let us fix coordinates $x_1,\ldots,x_n$ on the $n$-dimensional plane $\kk^{n}$.
The ring of $\infty$-jets at the origin $0\in\kk^{n}$ coincides 
with the ring of formal power series $\kk[[x_1,\ldots,x_n]]$ 
in
$n$ variables.
We will also use shorter notation $\calO_n$ for the same ring.
The infinite-dimensional Lie algebra of derivations of this ring is called the Lie algebra 
of formal vector fields on $\kk^n$ and is denoted by $W_n$.  
Any vector field $\nu\in W_n$ admits a presentation in coordinates $\nu=\sum_{i=1}^n \nu_i \dx{i}$, where $\nu_i\in \calO_n$.
The subspace of vector fields with linear coefficients $\left\{\sum_{i,j}c_{ij} x_i\dx{j}\right\}$ forms a Lie subalgebra
isomorphic to the Lie algebra $\gl_n=\gl_n(\kk)$ of $n\times n$-matrices. 

\subsubsection{Extension by $\g$-valued functions}
\label{sec::Wng}
To any Lie algebra $\g$ with 
commutator $[\ ,\ ]_\g$
we can assign a new infinite-dimensional Lie algebra
$\Wng$ (denoted later by $W(n;\g)$). 
This algebra is a semi-direct product of the Lie algebra of formal vector fields
and the Lie algebra of $\g$-valued formal power series on $\kk^{n}$.
In particular, the commutator in this algebra looks as follows:
\begin{equation}\label{eq::commutator}
[\eta_1+g_1\otimes p_1,\eta_2+g_2\otimes p_2]
=[\eta_1,\eta_2] + [g_1,g_2]_{\g} \otimes p_1 p_2 +
g_2\otimes \eta_1(p_2) - g_1\otimes \eta_2(p_1),
\end{equation}
where $\eta_i \in W_n$, $g_i \in \g$, $p_i \in \calO_n$ with $i\in \{ 1,2 \}$.
This algebra was used in~\cite{Khor::Wn_g} in order to define  
the characteristic classes of $G$-bundles (whenever $G$ is a compact group) 
using the constructions of formal geometry.

\subsubsection{Vector fields preserving foliation structures}
Let us substitute in the example from the previous Section~\ref{sec::Wng} 
instead of $\g$ the infinite-dimensional Lie algebra $W_m$ 
which does not admit a Lie group at all.
However, the geometrical meaning of the Lie algebra $\Wnm$ is clear.
Indeed, the Lie algebra $\Wnm$ is a subalgebra of the Lie algebra $W_{n+m}$ 
consisting of 
formal vector fields 
on $\kk^{n+m}$
 that preserve the trivial foliation $\kk^{n}\times \kk^{m}$, 
where
for any $p\in\kk^{m}$ the product $\kk^{n}\times \{p\}$ is considered to be a leaf of this foliation.
In other words, a vector field $\nu=\sum_{i=1}^{n+m} \nu_i\dx{i}$ belongs to $\Wnm$ whenever
for all $i$ from $1$ to $n$ the formal power series $\nu_i$ does not depend on $x_{n+1},\ldots,x_{n+m}$.

Let us also recall the generalization of this Lie algebra to the case of a flag of foliations.
Namely, let us fix a collection of integers $\dn=(n_0,n_1,\dots,n_k)$
and consider a collection of trivial embedded foliations 
$\{\F_1\supset\ldots\supset \F_k\}$ with prescribed codimensions.
Namely, $\F_i:= \kk^{n_0+\ldots+n_{i-1}}\times\kk^{n_i+\ldots+n_k}$
and each leaf of $\F_i$ is equal to $p\times\kk^{n_i+\ldots+n_{k}}$ for 
some point $p\in \kk^{n_0+\ldots+n_{i-1}}$.
The Lie algebra $W({n_0,n_1,\ldots,n_k})$ is a Lie subalgebra of $W_{n_0+n_1+\ldots+n_k}$
consisting of those vector fields that preserve the aforementioned flag of foliations.
In other words, a vector field $\nu=\sum_{i=1}^{n+m} \nu_i\dx{i}$ belongs to 
$W({n_0,n_1,\ldots,n_k})$ if and only if for all $i$ the formal power series 
$\nu_i$ depends only on variables $x_1,\ldots,x_{n_0+\ldots+n_r}$ where $r$ is the integer defined by the inequality
$n_0+\ldots+n_{r-1}< i < n_0+\ldots +n_{r}$.
The Lie algebra $W({n_0,n_1,\ldots,n_k})$ contains a direct sum of matrix subalgebras $\gl_{n_r}$ where  $r$ ranges from $0$ to $k$. 
Each $\gl_{n_r}$ is generated by the fields $x_i\dx{j}$ where $i,j$ satisfy 
 the inequality $n_0+\ldots+n_{r-1}< i,j < n_0+\ldots +n_{r}$.

\subsubsection{Vector fields that are linear in the normal direction to the leaves of a foliation}
\label{sec::Wn::def::WL}
We also want to give special notations for another class of subalgebras
of vector fields. These Lie algebras are very useful for our homological computations:\\
Define $WL({n|m})$ 
to be the 
subalgebra of $W({n,m})$ consisting of those 
fields $\nu=\sum_{i=1}^{n+m} \nu_i\dx{i}$ 
for which 
the power series $\nu_i$ is a polynomial of degree no more than $1$ 
 for $i=1,\dots,n$.
That is the intersection of $WL({n|m})$ with $W_n$ is a subspace of $W_n$ spanned by linear and constant vector fields.
Similarly, one can define the Lie algebra $WL({n_0,\ldots,n_{k}|n_{k+1},\ldots,n_{k+l})}$ by the same property.
We say that a vector field $\nu=\sum_{i=1}^{n_0+\ldots+n_k} \nu_i\dx{i}$ from $W({n_0,\ldots,n_{k+l}})$ belongs to 
$WL({n_0,\ldots,n_{k}|n_{k+1},\ldots,n_{k+l}})$ if all $\nu_i$ are linear or constant for all integer $i$ from the interval 
$[1,n_0+\ldots+n_{k}]$.

Analogously to 
the extension in
Subsection~\ref{sec::Wng},
 we consider 
the extension of the Lie algebra of 
flag-preserving vector fields 
by the Lie algebra of $\g$-valued functions.
Let $W(n_1,\ldots,n_k;\g)$ be the semi-direct product $W(n_1,\ldots,n_k)\ltimes\g\otimes \calO_{n_1+\ldots+n_k}$
with 
commutator defined analogously to~\eqref{eq::commutator},
while
the Lie algebra
 $WL(n_1,\ldots,n_k|n_{k+1},\ldots,n_{k+l};\g)$ is its subalgebra whose intersection with $W_{n_1+\ldots+n_k}$
consists of linear or constant vector fields.

\subsection{Lie algebra cohomology}
\label{sec::Lie::H::def}
The cohomology of a Lie algebra $\g$ with coefficients in a module $M$ is the derived homomorphism from a trivial module $\kk$ to $M$:
$$
H^{\udot}(\g;M):= Ext^{\udot}_{U(\g)}(\kk,M).
$$ 
However, the Koszul resolution $\Lambda^{\udot}(\g)\otimes U(\g) \rightarrow \kk$ produces a small complex that counts the derived homomorphisms.
This complex is called the Chevalley--Eilenberg complex $C^{\udot}(\g;M)$.
The CE complex is isomorphic to the tensor product $\Lambda^{\udot}(\g[1])^{*}\otimes M$ of a free graded commutative algebra generated by
the dual vector space $\g^{*}$ shifted by $1$ and the module $M$.
The differential is extended by the Leibniz rule 
from the maps $\g^{*}\rightarrow \Lambda^{2}\g^{*}$ and $M\rightarrow \g^{*}\otimes M$, dual to the commutator map
$[\ ,\ ]:\Lambda^2\g\rightarrow\g$ and the map $\g\otimes M \rightarrow M$ which defines the action.
For the trivial module $M=\kk$ the Chevalley--Eilenberg complex is a dg-algebra.
A review of the definition of homology of Lie algebras in terms 
of derived functors may be found in textbooks on homological algebra
(see e.g.~\cite{Weibel} {ch.7}).
We also refer the reader to the monograph \cite{Fuks::Lie::cohomology}
for detailed definitions via Chevalley-Eilenberg complex.
The same monograph also contains 
a
description of 
the
most frequently used methods of computation
and the particular computation for infinite-dimensional Lie algebras.

In this article we are mostly interested in the Lie algebra cohomology of infinite-dimensional Lie algebras.
There are several ways 
to avoid 
problems related to the finiteness conditions.
One way is to consider a grading on a Lie algebra such that the graded components are finite-dimensional.
Then all considered modules should be graded and the duals should 
also
be
graded duals. 
Another possibility is to consider the topology on a Lie algebra.
The Chevalley-Eilenberg complex $C^{\udot}(\g;L)$  of a topological Lie algebra $\g$
consists of continuous skew-symmetric maps from 
several copies of 
$\g$ to a topological module $L$.

The main computational method we use below is the Hochschild--Serre spectral sequences.
Let us briefly recall the origin of this spectral sequence.
Consider an embedding of Lie algebras $i:\h\hookrightarrow \g$ and a decomposition 
$\g\cong \h\oplus\g/\h$ as $\h$-modules. 
Then there exists a canonical filtration on the Chevalley-Eilenberg complex of the Lie algebra $\g$
associated with this embedding $i$:
\begin{equation}
\label{eq::Ser:Hoch:filtr}
C^{\udot}(\g;L) =Hom(\Lambda^{\udot}\g;L) \supset Hom(\Lambda^{\udot}\g\otimes\Lambda^{1}\g/\h; L) \supset \ldots
\supset Hom(\Lambda^{\udot}\g\otimes\Lambda^{k}\g/\h;L)\supset \ldots
\end{equation}
The spectral sequence associated with this filtration is called the Hochschild-Serre spectral sequence and has the following
first term $E_1^{p,q} = \rH^{q}(\h; Hom(\Lambda^{p}(\g/\h);L))$.

We do all homological computations over 
an arbitrary
field $\kk$ of zero characteristic and 
these computations 
remain 
true over the field of rational numbers $\mathbb{Q}$. 
However, all reasonable applications 
deal
with the case of complex or real coefficients
while working with formal geometry
for complex or real smooth manifolds respectively.

\subsection{Weyl superalgebra}
\label{sec::Weyl::super}
In this section we recall a standard construction from homological algebra.
The Weyl superalgebra described below is acyclic, but 
it
has a standard filtration
such that the corresponding spectral sequence coincides in some cases with the one coming from the universal bundle
over the classifying space. 
It is possible to define characteristic classes of fibrations using a filtered map from a Weyl superalgebra.
 The details of this construction may be found, for example, 
in~\cite{Fuks::Lie::cohomology,Khor::Wn_g}.

Let $\g$ be a Lie algebra. 
Consider a $\Z$-graded Lie superalgebra that is a semi-direct product of the Lie algebra $\g$ (placed in degree $0$) 
and its adjoint representation (placed in degree $-1$). Define a differential on this semi-direct product which maps identically one copy of $\g[1]$
 to another copy of $\g$. We denote this algebra by $\gtg$.
\begin{equation}
\label{eq::g_to_g}
\gtg :=
\begin{array}{ccccccccccccl}
 \ldots & \rightarrow & 0 & \rightarrow & \g & \stackrel{Id}{\rightarrow} & \g & \rightarrow & 0 & \rightarrow & \ldots 
& \text{ the complex itself }
\\
\ldots & & -2 & & -1 & & 0 & & 1 & & \ldots  & \text{ homological degrees} 
\end{array}
\end{equation}
The Chevalley-Eilenberg complex $C^{\udot}(\gtg)$ of this Lie super-algebra is called 
{\it the Weyl (super)algebra of $\g$} and is denoted by $\W^{\udot}(\g)$.
The Weyl algebra $\W^{\udot}(\g)$ is acyclic since the complex \rf{eq::g_to_g} has 
vanishing
cohomology. 
However, this complex $\W^{\udot}(\g)$ is a filtered complex and the corresponding spectral sequence is of 
particular interest.
Indeed, consider the Hochschild-Serre filtration associated with the embedding of the original Lie algebra $\g$ into $\gtg$.
We call this filtration {\it standard} and denote it by $F^{\udot}\W^{\udot}(\g)$.
In particular, we have the following isomorphisms for the associated graded terms:
$$
F^{2k}\W^{\udot}(\g) /F^{2k+1}\W^{\udot}(\g) = C^{\udot}(\g;S^{k}\g^{*})[-2k]
\quad \text{ and } \quad F^{2k-1}\W^{\udot}(\g) /F^{2k}\W^{\udot}(\g) = 0
$$
That is for even numbers the associated graded complex coincides with the Chevalley-Eilenberg complex of the original
Lie algebra $\g$ with coefficients in the symmetric power of the coadjoint representation.
The associated graded complexes for odd numbers of 
the
filtration are empty.

\begin{remark}
Suppose that $\g$ is a semi-simple Lie algebra associated with a compact Lie group $G$. 
The spectral sequence associated with the standard filtration on the Weyl algebra $\W^{\udot}(\g)$ is 
a stabdard example of \emph{transgression}
and coincides with the Leray-Serre spectral sequence associated with the universal fibration:
$EG\stackrel{G}{\rightarrow} BG$.
Here $EG$ is a contractible space endowed with the free action of the Lie group $G$ and the quotient space $BG$ is called the classifying space.
In particular, the cohomology of the classifying space $BG$ differs from zero only in even degrees and the even cohomology $\rH^{2n}(BG)$
coincides with the space of invariants $[S^{n}\g^{*}]^{\g}$ of the {$n$th} symmetric power of the coadjoint representation. 
\end{remark}

Similarly to the Chevalley-Eilenberg complex the Weyl algebra defines a contravariant functor from the category of Lie algebras
to the category of filtered dg-algebras. In particular, any morphism $\varphi:\h\rightarrow \g$ of Lie algebras defines
a morphism of corresponding Weyl superalgebras $\W(\varphi):\W^{\udot}(\g)\rightarrow \W^{\udot}(\h)$ 
compatible with standard filtrations.
Moreover, whenever $\varphi$ is an embedding and there exists an inverse map of $\h$-modules $\psi:\g\rightarrow \h$ 
(that gives an isomorphism
of $\h$-modules $\g\simeq \h\oplus \g/\h$), there exists 
a canonical map of filtred complexes  
\begin{equation}
\label{eq::weyl::map}
\psi^*:\W^{\udot}(\h)\rightarrow C^{\udot}(\g;\kk).
\end{equation}
Where the filtration on 
the
Weyl superalgebra is standard and the filtration on 
the
Chevalley-Eilenberg complex $C^{\udot}(\g;\kk)$
is the Hochschild-Serre filtration~\eqref{eq::Ser:Hoch:filtr} associated with the embedding $\h\hookrightarrow\g$.

\subsubsection{Relative case}

Suppose $\h$ is a Lie subalgebra of $\g$.
There is a notion of relative Lie algebra cohomology of the Lie algebra $\g$ relative to $\h$ and with coefficients in a $\g$-module $M$.
This cohomology theory can be described as a certain derived functor from the category of $\g$-modules with mild assumptions on the $\h$-action.
For example if $\h$ is semi-simple one has to consider derived homomorphisms in the category of $\g$-modules with $\h$ acting locally finite.
We refer for details on abstract definition to e.g.~\cite{Weibel}. 
Nevertheless, we define the relative Chevalley-Eilenberg complex $C^{\udot}(\g,\h;M)$ which computes the relative cohomology.
This complex consists of those chains $c\in C^{\udot}(\g;M)=\Hom(\Lambda^{\udot}\g,M)$
such that $c(h,\ldots)= 0= d_{CE}(c)(h,\ldots)$ whenever $h\in\h$.
In particular, if $h$ is semisimple then $C^{\udot}(\g,\h;M)$ is isomorphic to the subcomplex of $\h$-invariants
 $\Hom_{\h}(\Lambda^{\udot}(\g/\h),M)$.

We can apply the same procedure for the Weyl superalgebra.
Indeed, starting from a Lie subalgebra $\h\hookrightarrow \g$ we define {\it the relative Weyl superalgebra $\W^{\udot}(\g,\h)$} 
to be the relative Chevalley-Eilenberg complex of the Lie superalgebra $\gtg$ relatively to the Lie subalgebra
$(0\to\h)$.
In particular, the components of the associated graded complex with respect to the standard filtration looks as follows:
\begin{equation}
\label{eq::Weyl::assoc::grad}
F^{2k}\W^{\udot}(\g,\h) /F^{2k+1}\W^{\udot}(\g,\h) = C^{\udot}(\g,\h;S^{k}\g^{*})[-2k] \quad
\text{ and } \quad F^{2k-1}\W^{\udot}(\g,\h) /F^{2k}\W^{\udot}(\g,\h) = 0
\end{equation}
Note that 
the
relative Weyl dg-algebra is no 
longer
acyclic.
 However, if $\h$ is semi-simple and admits a compact group $H$
 then the cohomology of the relative Weyl algebra $\W^{\udot}(\g,\h)$ is equal to the 
cohomology of the classifying space $BH \cong EH/H \cong\{point\}/H$. 
That is the total cohomology $\rH^{\udot}(\W^{\udot}(\g,\h))$ 
vanishes in odd degrees and is is isomorphic
to the ring of invariants $[S^{\udot}\h]^{\h}$ in even degrees.

\section{Relative chains and truncated Weyl superalgebras}
\label{sec::chains::Weyl::Wn}
In this section  
we begin by recalling
the classical results from~\cite{GF} concerning the description of the relative cohomology ring
of the Lie algebra $W_n$ modulo 
the
subalgebra of linear vector fields $\gl_n$ and 
similar results obtained by the author in~\cite{Khor::Wn_g} 
on 
the
cohomology ring of the Lie algebra $\Wng$.
Then we consider a particular case of the Lie algebras $W(m,n)$ and state similar results for the Lie algebra $WL(m|n)$.
Finally we state the main general results which we prove in the next Section~\ref{sec::main::methods}.
All results in this section either consist of careful 
descriptions
 of relative chains
based on hunting of $\gl$-invariants or are implications of theorems from the 
following section.
In either case
we will give direct links 
to Appendix~\ref{sec::chains::gl::inv} or 
to Section~\ref{sec::main::methods}.

\subsection{known results}
\label{sec::Wn::known}
Recall one of the first cohomological 
computations
in formal geometry due to Gelfand and Fuchs:
\begin{theorem}{(\cite{GF})}
\label{thm::HWn}
The space of relative cochains 
of formal vector fields $W_n$ with constant coefficients
is isomorphic to the truncated cohomology of the classifying space $BU_n$:
$$
\frac{\rH^{\udot}({BU_n})}{\rH^{>2n}(BU_n)} = 
\frac{[S^{\udot}(\gl_n)]^{\gl_n} }{[S^{>n}(\gl_n)]^{\gl_n} }
= 
\frac{\W^{\udot}(\gl_n,\gl_n) } {F^{2n+1}\W^{\udot}(\gl_n,\gl_n)}
\stackrel{\cong}{\longrightarrow} C^{\udot}(W_n,\gl_n;\kk)
$$
In particular, there are no relative cochains of odd degree.
\end{theorem}
This computation is based on the description of relative chains as $\gl_n$-invariants.
We refer to the original paper~\cite{GF} and to 
Appendix~\ref{sec::chains::Wn::graphs} for details.
We draw the arrow from the truncated Weyl superagebra following the existence of map~\rf{eq::weyl::map}.
A
 similar computation with $\gl_n$-invariants was done by the author in~\cite{Khor::Wn_g} and leads to the following result:
\begin{theorem}{(\cite{Khor::Wn_g})}
\label{thm::old::WL(0|n;g)}
The ring of relative chains of the Lie algebra of formal vector fields extended by $\g$-valued functions 
coincides with the quotient of the relative Weyl superalgebra $\W^{\udot}(\gl_n\oplus \g,\gl_n)$ modulo 
the 
$2n+1$
part of the standard filtration:
$$
 \W^{\udot}(\gl_n\oplus \g,\gl_n)/F^{2n+1}\W^{\udot}(\gl_n\oplus \g,\gl_n)
\stackrel{\cong}{\longrightarrow} C^{\udot}(\Wng,\gl_n;\kk).
$$
\end{theorem}
We will go further and state the analogous result for the Lie subalgebra of constant and linear vector fields
extended by $\g$-valued functions. 
In Section~\ref{sec::Wn::def::WL} these algebras 
are denoted by $WL(m|0;\g)$ or $WL(m|;\g)$ for simplicity.
\begin{theorem}
\label{thm::old::WL(n|0;g)}
 The relative chain complex of the Lie algebra $WL(m|0;\g)$ is isomorphic to the truncated Weyl superalgebra 
of the Lie algebra $\g$:
\begin{equation}
\label{eq::WL_g=Weyl}
 \W^{\udot}(\g)/F^{2m+1}\W^{\udot}(\g)
\stackrel{\cong}{\longrightarrow} C^{\udot}(WL(m|0;\g),\gl_m;\kk).
\end{equation}
\end{theorem}
In particular, we can substitute 
for $\g$ the Lie algebra of vector fields on $\kk^{n}$ 
 and get the following:
\begin{equation}
\label{eq::W_m::WL_nm}
 \W^{\udot}(W_n,\gl_n)/F^{2n+1}\W^{\udot}( W_n,\gl_n)
\stackrel{\cong}{\longrightarrow} C^{\udot}(WL(m|n),\gl_m\oplus\gl_n;\kk).
\end{equation}
This observation becomes crucial for further implications.
Namely, the computation of the cohomologies of the Lie algebras $WL(m|n)$ for different $m$ will be enough
in order to get the description of the spectral sequence associated with the standard filtration on the 
relative Weyl algebra $\W^{\udot}(W_n,\gl_n)$. 

\subsection{Main theorems}
\label{sec::main::res}
Denote by $\pi$ the natural surjection $WL(m|n)\twoheadrightarrow W_n\twoheadrightarrow \gl_n$ of $\gl_n$-modules.
\begin{theorem}
\label{thm::homol::WLmn}
 The morphism 
of
Weyl superalgebras associated with the projection $\pi$ leads to the following  quasi-isomorphism
of dg-algebras:
\begin{multline}
\rH^{\udot}(BU_n;\kk)/ \rH^{\geq 2(n+m)+1}(BU_n;\kk) = 
\W^{\udot}(\gl_n,\gl_n)/F^{2(n+m)+1}\W^{\udot}(\gl_n,\gl_n) \stackrel{quis}{\longrightarrow} \\
\stackrel{quis}{\longrightarrow} C^{\udot}(WL(m|n),\gl_m\oplus\gl_n;\kk).
\end{multline}
In other words, the relative cohomology of the Lie algebra $WL(m|n)$ is isomorphic to the truncated ring
of polynomials $\kk^{\udot\leq 2(m+n)}[\Psi_2,\ldots,\Psi_{2n}]$ 
of degrees less 
than
or equal to 
$2m+2n$,
 and
generated by the $n$ even variables
$\Psi_2,\ldots,\Psi_{2n}$ with $deg\Psi_{2i}=2i$. 
\end{theorem}
\begin{proof}
 The proof of this theorem is postponed to Section~\ref{sec::main::methods} 
In this section we will only explain why this result implies all 
the
others. 
\end{proof}

Together with Equation~\eqref{eq::W_m::WL_nm} we get the main corollary of this paper:
\begin{corollary}
\label{thm::Wn:SWn:relative}
For 
any
$m\geq 1$,  the relative cohomology of the Lie algebra $W_n$ 
with coefficients in the $m$-th symmetric power of the coadjoint representation,
is always
different from zero only in 
the degree
$2n$, and 
 this is  the same as the space of invariants $[S^{m+n}(\gl_n)]^{\gl_n}$: 
$$
\rH^{2n}(W_n,\gl_n;S^{m}(W_n^{*})) =  [S^{m+n}(\gl_n)]^{\gl_n}, \ \text{ and } \ 
\rH^{i\neq 2n}(W_n,\gl_n;S^{m}(W_n^{*})) = 0
$$
\end{corollary}
\begin{proof}
The isomorphism~\eqref{eq::W_m::WL_nm} between the Chevalley-Eilenberg complex of the Lie algebra $WL(m|n)$ and
the truncated Weyl superalgebra for $W_n$ implies the following 
short exact sequence of complexes that relates two 
consecutive
values of $m$:
$$
0\to \frac{F^{2m-1}\W^{\udot}(W_n,\gl_n)}{F^{2m+1}\W^{\udot}(W_n,\gl_n)} 
\to C^{\udot}(WL(m|n),\gl_m\oplus\gl_n;\kk) 
\stackrel{p_m}{\to} C^{\udot}(WL(m-1|n),\gl_{m-1}\oplus\gl_n;\kk) 
\to 0
$$
where the surjective map has been denoted $p_m$.
As was mentioned in~\eqref{eq::Weyl::assoc::grad}
the standard filtration on the Weyl algebra has 
$F^{2m}=F^{2m-1}$, 
and the associated graded $F^{2m}/F^{2m+1}$ 
is isomorphic to the Chevalley-Eilenberg complex with coefficients in the $m$-th symmetric power of the coadjoint representation. 
Theorem~\ref{thm::homol::WLmn} implies that 
the 
relative cohomology of the Lie algebra $WL(m|n)$ is the truncation
of the polynomial ring
$\rH^{\udot}(BU_n)=\kk[\Psi_2,\ldots,\Psi_{2n}]$. Namely we cut off the polynomials of degree $2(m+n+1)$.
Therefore, the map $p_m$ is surjective on cohomology rings and is also the  same as cutting off the polynomials of degree $2(m+n)$.
Hence, the map $p_{m}$ should be surjective on cohomology rings and we get a short exact sequence of cohomologies: 
$$
\xymatrix{
 \rH^{\pt}(W_n,\gl_n;S^{m}W_n^*)[2m] \ar@{^{(}->}[r] &
\rH^{\udot}(WL(m|n),\gl_m\oplus\gl_n;\kk) 
\ar@{->>}^(.45){p_{m}}[r] 
\ar@{->}^{\cong}[d]
&
\rH^{\udot}(WL(m-1|n),\gl_{m-1}\oplus\gl_n;\kk) 
\ar@{->}^{\cong}[d]
 \\
 & \kk^{\udot\leq 2(m+n)}[\Psi_2,\ldots,\Psi_{2n}] 
\ar@{->>}[r] 
& \kk^{\udot\leq 2(m-1+n)}[\Psi_2,\ldots,\Psi_{2n}]
}
$$
Thus, we conclude
that the shifted cohomology of the cochain complex $C^{\udot}(W_n,\gl_n;S^{m}W_n^*)[2m]$ is different from zero 
only 
at
degree $2(m+n)$, where it coincides with the polynomials of degree $2(m+n)$.
Making the opposite homological shift we 
obtain
\begin{equation}
\label{eq::H(Wn,SWn)}
\rH^{i}(W_n,\gl_n;S^{m}W_n^*) = 
\left\{
\begin{array}{c}
{
\begin{array}{c}
{
\rH^{2(m+n)}(BU_n;\kk) = 
[S^{m+n}(\gl_n)]^{\gl_n} = }\\
\{\text{polynomials of degree } 2(n+m) \text{ in } 
\kk[\Psi_2,\ldots,\Psi_{2n}] \},
\end{array}
}
 \text{ if } i=2n, 
\\
0, 
\text{ if } i\neq 2n.
\end{array}
\right.
\end{equation}
\end{proof}

We also state here the generalization of Theorem~\ref{thm::homol::WLmn}
to the Lie algebra related to a flag of foliations with given codimensions.
\begin{theorem}
\label{thm::HWm|nnng}
For any 
collection of integers  $(m;n_1,\ldots,n_k)$, 
 for an arbitrary Lie algebra $\g$
and a ($\gl_{m}\oplus\gl_{n_1}\oplus\ldots\oplus\g$)-equivariant projection 
$\pi:WL(m|n_1,\ldots,n_k;\g)\twoheadrightarrow \gl_{n_1}\oplus\ldots\gl_{n_{k}}\oplus\g$ {\cbl,*}
the corresponding map from the Weyl superalgebra factors through
the quotient by ideal
$I_{n_1,\ldots,n_k}(g)$, which is generated by symmetric powers 
$ 
S^{(m+1+\sum_{i=1}^{j}n_{i})}(\gl_{n_1}\oplus\ldots\oplus\gl_{n_{j}})
$ 
with $j$ ranging from $1$ to $k$
and the symmetric power 
$ 
S^{(m+1+\sum_{i=1}^{k}n_{i})}(\gl_{n_1}\oplus\ldots\oplus\gl_{n_{k}}\oplus \g).
$ 
The quotient map is a quasi-isomorphism:
\begin{equation}
\xymatrix{
\frac{\W^{\udot}(\gl_{n_{1}}\oplus\ldots\oplus\gl_{n_{k}}\oplus\g,\gl_{n_{1}}\oplus\ldots\oplus\gl_{n_{k}})}{I_{n_1,\ldots,n_k}(\g)}
\ar@{=}[d]
\ar@{->}^(.40){\quis}[r]
\ar@{->}^{\quis}[rd]
& 
{\W^{\udot}(W(n_1,\ldots,n_k;\g),\gl_{n_1}\oplus\ldots\oplus\gl_{n_k})}/{F^{2m+1}}
\ar@{->}^{\quis}[d]
\\
\frac{\rH^{\udot}(BU_{n_1}\times\ldots\times BU_{n_k})\otimes \W^{\udot}(\g)}{I_{n_1,\ldots,n_k}(\g)}
&
C^{\udot}(WL(m|n_1,\ldots,n_k;\g),\gl_{m}\oplus\ldots\gl_{n_k};\kk).
} 
\label{eq::Wg=WLnng}
\end{equation}
Moreover, 
quasi-iso~\eqref{eq::Wg=WLnng} is compatible with filtrations: 
the standard
filtration on 
the Weyl algebra and the Hochschild-Serre filtration
on the Chevalley-Eilenberg complex.
Consequently, the same result remains valid in the non-relative case (see Theorem~\ref{thm::HWnm::absolute}).
\end{theorem}
\begin{proof}
The proof repeats 
that
of Theorem~\ref{thm::homol::WLmn}.
However, one may prove this theorem by induction on the number $k$ of foliations in the flag.
In this case, the
key ingredient will be  
the computation~\eqref{eq::H(Wn,SWn)} and 
the subsequent degeneration of 
Hochschild-Serre spectral sequences based on the following embeddings of Lie algebras:
$$
WL(m|n_1,\ldots,n_k;\g)\supset WL(m|n_1,\ldots,n_k) \supset \ldots \supset WL(m|n_1).
$$
We omit 
a description of these spectral sequences which is absolutely straitforward.
\end{proof}
In particular, Theorem~\ref{thm::HWm|nnng} describes the cohomology ring
of the Lie algebra of vector fields preserving a given flag of foliations.
\begin{corollary}
\label{thm::relative_har_flag_sloj}
There exists an isomorphism between the truncated ring of characteristic classes and the relative cohomology ring of the Lie 
algebra of vector fields preserving a given flag of foliations:
$$
\frac{\rH^{\udot}(BU_{n_1}\times\ldots\times BU_{n_k})}{I_{n_1,\ldots,n_k}} \stackrel{\quis}{\longrightarrow}
\rH^{\udot}(W(n_1,\ldots,n_k),\gl_{n_1}\oplus\ldots\oplus\gl_{n_k};\kk),
$$
where the ideal $I_{n_1,\ldots,n_k}$ is generated by the union of subspaces
$\rH^{>2(n_1+\ldots+n_r)}(BU_{n_1}\OP\ldots\OP BU_{n_r})$ for $r=0,\dots,k$.

In particular, we have the following equivalences of vector spaces indexed by the homological degree $i$
$$\rH^{i}(W(n_1,\ldots,n_k),\gl_{n_1}\oplus\ldots\oplus\gl_{n_k};\kk) =
\left\{
\begin{array}{l}
{ 0,} {\text{ if } i \text{ is odd,}} \\
{\left[\frac{S^{l}(\gl_{n_1}\oplus\ldots\oplus\gl_{n_k})}
{\left(\prod\limits_{j=1}^{k} S^{n_1+\ldots+n_j+1}(\gl_{n_1}\oplus\ldots\oplus\gl_{n_j}) \right)}\right]^{\gl_{n_1}\oplus\ldots\oplus\gl_{n_k}}, 
}{\text{ if } i =2l.}
\end{array}
\right.
$$
The union of these identities gives a graded isomorphism of corresponding rings.
\end{corollary}
\begin{proof}
 Substitute $m=0$ and $\g=0$ in Theorem~\ref{thm::HWm|nnng}.
\end{proof}

\subsection{Absolute case}
\label{sec::W_flag}
So far we have been discussing only relative cohomologies 
because we think that this is the core of the construction.
Below we explain what should be done in order to compute the absolute cohomology.
The general prescription looks as follows:
in all statements from 
Section~\ref{sec::main::res}
one has to replace the relative Weyl algebra for matrix Lie algebras by the absolute 
ones.
For example, all  the maps in Diagram~\eqref{eq::Wg=WLnng} 
maintain the property of being quasi-isomorphisms
if one replaces 
the relative complexes by absolute 
ones.

Here are some
corollaries which we find important for applications in the geometry of foliations:
\begin{theorem}
\label{thm::sym_coef}
\label{thm::Wn:SWn:absolute}
The absolute cohomology of the Lie algebra of formal vector fields $W_n$ with coefficients in the 
$m$th
symmetric 
power ($m\geq 1$) of the coadjoint representation is the tensor product of the relative cohomology and the cohomology 
$\rH^{\udot}(\gl_n;\kk)$ of the matrix Lie algebra:
$$\rH^{i}(W_n;S^{m}W_n^{*}) =
\left\{ \begin{array}{l}
         {[S^{n+m}\gl_n]^{\gl_n}\otimes [\Lambda^{i-2n}(\gl_n)]^{\gl_n}, {\mbox{ if }} 2n\leq i\leq n^2+2n,}\\
         {0,  {\mbox{ otherwise. }}}
       \end{array}
\right.
$$
\end{theorem}
\begin{proof}
Consider the Hochschild-Serre spectral sequence associated with the canonical embedding $\gl_n\hookrightarrow W_n$.
The first term of this sequence is $E_1^{p q} = \rH^{q}(\gl_n;\kk)\otimes \rH^{p}(W_n,\gl_n;S^{k}W_n^{*})$.
Computation~\ref{eq::H(Wn,SWn)} implies that for $p\neq 2n$ the corresponding space $E_1^{p q}$ is zero and, therefore, 
this spectral sequence degenerates in the first term.
\end{proof}

\begin{theorem}
\label{thm::HWnm::absolute}
\label{thm::W_flag_absolute}
 The absolute cohomology of the Lie algebra of vector fields preserving a given flag of foliations
may be computed via the cohomology of the truncated Weyl algebra. We have a quasi-isomorphism:
$$
\frac{\W^{\udot}(\gl_{n_1}\oplus\ldots\oplus\gl_{n_k})} {I_{n_1,\ldots,n_k}} \stackrel{\quis}{\longrightarrow} 
C^{\udot}(W(n_1,\ldots,n_k);\kk)
$$
where the ideal $I_{n_1,\ldots,n_k}$ is generated by symmetric powers
$ 
S^{(1+\sum_{i=1}^{j}n_{i})}(\gl_{n_1}\oplus\ldots\oplus\gl_{n_{j}})
$ with $j$ 
ranging
from $1$ to $k$. 
\end{theorem}
\begin{proof}
 The implication from the relative case (Corollary~\ref{thm::relative_har_flag_sloj}) is standard. 
The method used  in~\cite{GF,Khor::Wn_g} for the same implications related with the cohomology of  Lie algebras
$W_n$ and $W(n;\g)$ respectively is also applied in our situation.
\end{proof}

In Section~\ref{sec::full_flag_formulas}
we will show a particular computation of the generating series of the cohomology 
of 
the
truncated Weyl algebra for the case $n_i=1$ for all $i$.
From the geometrical point of view this cohomology corresponds to characteristic classes of 
a full flag of foliations.

\subsection{Characteristic classes of flags of foliations}
\label{sec::obstruction}
In this subsection we first recall the standard construction of characteristic classes of 
foliations and of flags of foliations and then show how does our cohomological computation 
predicts a cohomological obstruction 
to the existence of
 a flag of foliations. 

As was mentioned in the introduction the concept of formal geometry
assigns to a foliation $\F$ of codimension $n$ on a smooth manifold $X$ a characteristic map
$$
ch: \rH^{\udot}(W_n,\mathsf{o}(n);\R) \longrightarrow \rH^{\udot}_{DR}(X)
$$
where $\mathsf{o}(n)$ is the Lie algebra of the group of orthogonal matrices.
If the foliation $\F$ is framed, i.e. the trivialization of the normal bundle is fixed,
then the corresponding characteristic classes come from the characteristic map 
$$
ch: \rH^{\udot}(W_n;\R) \longrightarrow \rH^{\udot}_{DR}(X)
$$
with the source space being isomorphic to the absolute cohomology of the Lie algebra $W_n$.
For a foliation of codimension $1$, the relative and absolute characteristic maps coincide and 
lead to a definition of the Godbillon-Vey class (\cite{Haefliger},\cite{Bott_har_cl}).
Therefore, we call 
the representatives of
 the absolute cohomology by generalized Godbillon-Vey classes. 

The same construction assigns to a flag of foliations of codimensions 
$(n_1,\ldots,n_k)$
The codimension is considered in the nested sense. 
That is the foliation $\F_i$ is a subfoliation of $\F_{i-1}$ and has codimension $n_1+\ldots+n_i$
on a smooth manifold $X$
a characteristic map
$$
ch: \rH^{\udot}(W(n_1,\ldots,n_k),\mathsf{o}(n_1)\oplus \ldots \mathsf{o}(n_k);\R)
\longrightarrow \rH^{\udot}(X)
$$
and whenever all foliations are framed we get the characteristic map from the absolute Lie algebra cohomology:
$$
ch: \rH^{\udot}(W(n_1,\ldots,n_k);\R)
\longrightarrow \rH^{\udot}(X)
$$
The construction of characteristic maps becomes 
 more explicit whenever 
a foliation is determined by a system of \emph{determining forms}:
\begin{definition}
The system of smooth $1$-forms $\omega_1,\ldots,\omega_n$ on a manifold $X$ determines a framed foliation of codimension $n$ 
iff the following conditions are satisfied:
\begin{itemize}
 \item[($\imath$)] The covectors $\omega_1(p)$,\ldots,$\omega_n(p)$ are linearly independent at each point $p\in X$.
Equivalently, the $n$-form $\omega_1\wedge\ldots\wedge\omega_n$ 
is never zero
on $X$.
 \item[($\imath\imath$)] 
For all $i$ there exists a collection of $1$-forms $\eta_{ij}$
such that $d\omega_i = \sum_{j=1}^{n} \omega_j\wedge \eta_{ij}$, where $d$ is the De Rham differential.
Equivalently, the product $\omega_1\wedge\ldots\wedge\omega_n\wedge d\omega_i$ is zero for all $i=1$,\ldots,$n$.
\end{itemize} 
\end{definition}
If the system $\omega_1,\ldots,\omega_n$ is a system of determining forms of a foliation $\F$ on a manifold $X$,
then the restriction of $\omega_i$ on any component of an intersection of a leaf of $\F$ 
and any open contractible set on $X$ is identically zero for all $i$.
Note, that the system of determining forms $\omega_1,\ldots,\omega_n$ 
defines a flag of foliations $\{\F_1\supset\ldots\supset\F_k\}$
of nested codimensions $n_1,n_2,\ldots,n_k$ with $\sum_{i=1}^{k}n_i=n$ iff 
for all $r=1$,\ldots,$k$ the collection $\omega_1$,\ldots,$\omega_{n_1+\ldots+n_r}$ 
is a system of determining forms of the foliation $\F_r$.
This is equivalent to the following condition:
$$
\forall i\leq \sum_{j=1}^{r} n_r \text{ there exists } 
\eta_{ij}\in\Omega^{1}_{DR}(X) \text{ such that } d\omega_i=\sum_{j=1}^{n_1+\ldots+n_r} \omega_j\wedge \eta_{ij}.
$$
The matrix units $e_{ij}$ with $1\leq i,j\leq n$ form a basis of the Lie algebra $\gl_n$.
We define a linear map 
$$ch_{\omega}:\gl_{n_1}\oplus\ldots\oplus\gl_{n_k} \rightarrow \Omega^{1}_{DR}(X)$$
by mapping 
the
matrix unit $e_{ij}$ from the 
$r$th factor $\gl_{n_r}$ to the $1$-form 
$\eta_{n_1+\ldots+n_{r-1}+i,n_1+\ldots+n_{r-1}+j}$.
Recall that the Weyl superalgebra $\W^{\udot}(\gl_{n_1}\oplus\ldots\oplus\gl_{n_k})$ 
is a free acyclic skew-commutative algebra, therefore, any linear map of 
the generators $e_{ij}$ is
extended in a unique way to 
a 
map of dg-algebras:
$$
ch_{\omega}:\W^{\udot}(\gl_{n_1}\oplus\ldots\oplus\gl_{n_k}) \longrightarrow \Omega_{DR}^{\udot}(X)
$$
\begin{statement}
The map $ch_{\omega}$ assigned to a system of determining forms of a flag of foliations on a manifold 
factors through the truncated Weyl superalgebra and determines characteristic classes of this flag of foliations:
$$
\xymatrix{
\W^{\udot}(\gl_{n_1}\oplus\ldots\oplus\gl_{n_k}) 
\ar@{->>}[r]
&
 \frac{\W^{\udot}(\gl_{n_1}\oplus\ldots\oplus\gl_{n_k})}{I_{n_1,\ldots,n_k}}
\ar@{->}^{ch_{\omega}}[rd]
\ar@{->}_{quis}[d] & \\
& C^{\udot}(W(n_1,\ldots,n_k);\R) 
\ar@{->}^(.6){ch}[r] 
&
\Omega_{DR}^{\udot}(X)
}
$$
where the ideal $I_{n_1,\ldots,n_k}$ and the vertical quasi-iso 
are the same as
were defined in Theorem~\ref{thm::HWnm::absolute}.
\end{statement}
\begin{remark}
 The construction of characteristic classes via a system of determining forms does not require
any knowledge of infinite-dimensional Lie algebras and their cohomology.
(See e.g.~\cite{Foliations::Candel} ch.6.)
However, the proofs and the exposition 
become clear 
when
working with  the Gelfand-Fuchs cohomology.
See e.g.~\cite{Haefliger} for the comparison of these two languages.
\end{remark}

Let us fix two integers $m$ and $n$.
The commutative diagram of embeddings of Lie algebras:
$$
\xymatrix{
\gl_m \ar@{^{(}->}[r] \ar@{^{(}->}[d] & \gl_n\oplus\gl_m \ar@{^{(}->}[r] \ar@{^{(}->}[d] & \gl_{m+n} \ar@{^{(}->}[d] \\
W_m \ar@{^{(}->}[r] & W(m,n) \ar@{^{(}->}[r] & W_{m+n}
}
$$
defines the following commutative diagram of dg-algebras:
$$
\xymatrix{
\W^{\udot}(\gl_m)/ I_{m}  \ar@{->}[d]^{\quis} & 
\frac{\W^{\udot}(\gl_n\oplus\gl_m)}{I_{m,n}} \ar@{->}[l] \ar@{->}^{\quis}[d] &
\W^{\udot}(\gl_{m+n})/ I_{m+n} \ar@{->}^{\quis}[d] \ar@{->}[l]\\
C^{\udot}(W_m;\kk) & C^{\udot}(W(m,n);\kk) \ar@{->}[l] & C^{\udot}(W(m+n);\kk) \ar@{->}[l]
}
$$
This construction is universal and is compatible with 
 characteristic maps of foliations.
Indeed, for a flag  consisting of 
just
 two foliations $\F'\supset \F$ on a manifold $X$
of codimensions $m$ and $m+n$ respectively,
 we have the following commutative diagram:
\begin{equation}
\label{diag::char::Wn:X} 
\xymatrix{
\rH^{\udot}(W_m;\R)  \ar@{->}_{ch(\F')}[dr] 
& \rH^{\udot}(W(m,n);\R)  \ar@{->}|{ch(\F'\supset \F)}[d] \ar@{->}[l] 
& \rH^{\udot}(W(m+n);\R)  \ar@{->}^{ch(\F)}[dl] \ar@{->}[l] \\
& \rH^{\udot}(X) &
}
\end{equation}
\begin{corollary}
\label{cor::non::integr}
 A foliation $\F$ on a manifold $X$ of codimension $n$ may not be included into a flag of foliations $\F'\supset \F$
if it has a nonzero characteristic class of degree greater than $n^2+2$.
\end{corollary}
\begin{proof}
Let us count the top possible degrees of elements in the truncated Weyl superalgebras.
Indeed, elements of maximal degree in $\frac{\W^{\udot}(\gl_n)}{I_n}$ 
belong
 to the subspace
$\Lambda^{top}(\gl_n^{*})\otimes S^{n}(\gl_n)$. Therefore, $\rH^{i}(W_n;\kk)=0$ for 
$i> n^{2}+2n$.
Also by 
using degree 
arguments 
 we know that $\rH^{i}(W(n-d,d);\kk)=0$ for $i>(n-d)^{2} + d^2 +2n$.
The 
number 
$(n-d)^{2} + d^2 +2n$ for $d=1$,\ldots,$n-1$ 
reaches its maximum
when $d=1$ or $d=n-1$ 
where it takes the value

$n^2+2$.
Consequently, 
$$\forall i>n^2+2 \text{ and } \forall d<n \quad \rH^{i}(W(n-d,d);\kk)=0.$$
Therefore, 
existence of a 
nonzero element in the image of $ch(\F)$ of homological degree greater than $n^{2}+2$ 
is in contradiction
with the 
 commutative diagram~\eqref{diag::char::Wn:X} 
\end{proof}

Let us show that the example of a foliation mentioned in the introduction satisfies the conditions of Corollary~\ref{cor::non::integr}.
Let $\Gamma$ be a discrete subgroup of $SL_{n+1}(\R)$ such that the quotient space $SL_{n+1}(\R)/\Gamma$ is compact.
Let $\mathsf{P}$ be a subgroup of $SL_{n+1}(\R)$ that fixes a given line in $\R^{n+1}$.
Any orbit of the left action of $\mathsf{P}$ on the quotient space $SL_{n+1}(\R)/\Gamma$ 
defines a leaf of a foliation of codimension $n$ 
in
 $SL_{n+1}(\R)/\Gamma$. Denote 
this
foliation by $\F_{\mathsf{P}}$.
\begin{statement}
\label{cor::SL::foliation}
The foliation $\F_{\mathsf{P}}$ on the compact space $SL_{n+1}(\R)/\Gamma$ has a nontrivial characteristic class of the top degree $n^2+2n$
and, therefore, for $n\geq 2$ may not be a subfoliation. 
\end{statement}
\begin{proof}
The main feature of examples of foliations coming from the action of a Lie subgroup $H$ on a compact quotient 
of a Lie group $G$ by a discrete subgroup $\Gamma$ is that most of 
the 
computations 
inherit
the 
structures of 
the 
computations for the corresponding Lie algebras ${\mathsf{p}}=Lie(\mathsf{P}) \subset \g=Lie(G)$.
Standard arguments shows that 
there 
exists a collection of $1$-forms $\omega_{i,j}\in \Omega^{1}(SL_{n+1}(\R)/\Gamma)$ 
where $i,j$ range from $1$ to $n+1$,
$\sum_{i=1}^{n+1} \omega_{i,i}=0$ and 
$d\omega_{i,j} = \sum_{k=1}^{n+1}\omega_{i k} \omega_{k j}$.
The foliation $\F_{\mathsf{P}}$ is defined by the system of determining forms $\omega_{1,n+1}$,\ldots,$\omega_{n,n+1}$.
The direct check shows that there exists a class in $\rH^{n^{2}+2n}(W_n;\R)$ whose image under 
 the characteristic map is equal to the
product $\wedge_{(i,j)\neq (n+1,n+1)}\omega_{i,j}\in\Omega^{(n+1)^{2}-1}(SL_{n+1}(\R)/\Gamma)$. 
The latter product is a volume form
on the space $SL_{n+1}(\R)/\Gamma$ and, therefore, represents a nontrivial cohomological class.
\end{proof}

\section{The core of the proof}
\label{sec::main::methods}
Let us 
begin by stating 
several technical results and 
then 
showing 
how 
they imply Theorem~\ref{thm::homol::WLmn}.
The idea is to compute the relative cohomology of the Lie algebra $W_{m+n}$ of all vector fields and then show that 
the inclusion $WL(m|n)\hookrightarrow W_{m+n}$ produces the surjective map of relative cohomology rings in 
the opposite direction.
The surjectivity follows from the degenerations of the Hochschild-Serre spectral sequences associated with
embeddings of a parabolic subalgebra. 
Let us start by looking at some 
general results 
on cohomologies of parabolic subalgebras.

\subsection{Parabolic subalgebra}
\label{sec::parabolic::gl}
Consider the Lie algebra of matrices $\gl_{m+n}$.
Let $\b$ be it's parabolic subalgebra,
 with maximal reductive subalgebra isomorphic to the  
direct sum 
$\gl_m\oplus\gl_n$.
Note that $\b$ is embedded into the Lie algebra $WL(m|n)$,
and, moreover, 
$\b$ is isomorphic to the intersection $WL(m|n)\cap \gl_{m+n}$ 
inside $W_{m+n}$.
In this subsection we 
shall present 
several
general homological results on relative homology of 
 the parabolic subalgebra $\b\hookrightarrow\gl_{m+n}$,
with all proofs 
 postponed to Appendix~\ref{sec::parabolic}.

Denote the vector space $\kk^{m}$ by $V$ and the vector space $\kk^{n}$ by $U$,
such that $\gl_m=\gl(V)$, $\gl_n=\gl(U)$, $\gl_{m+n}=\gl(V\oplus U)$ and
the parabolic subalgebra $\b$ is isomorphic to $\gl(V)\oplus\gl(U)\oplus V^{*}\otimes U$.
Recall, that the set of irreducible $\gl_n=\gl(U)$-modules are 
enumerated
by dominant highest weights
$\lambda=\{\lambda_1\geq\ldots\geq \lambda_n\}$ with all $\lambda_i\in\mathbb{Z}$.
Moreover, if  $\lambda_n$ and, consequently, all $\lambda_i$ are nonnegative integers then the collection $\lambda$ is called a Young diagram
and  the corresponding 
irreducible $\gl_n$ module is called polynomial (or holomorphic). 
Moreover, this module may be defined by using Schur-Weyl duality:
Let $|\lambda|=\sum \lambda_i$ be the number of boxes in the Young diagram $\lambda$.
Let $\Sigma_{\lambda}$ be an irreducible representation of the symmetric group $S_{|\lambda|}$ 
assigned to the same diagram $\lambda$.
Then the irreducible polynomial $\gl(U)$-module with the highest weight $\lambda$ is isomorphic to the image of the Schur functor
\begin{equation}
\label{eq:Schur::Functor}
S^{\lambda}(U):= U^{\otimes |\lambda|}\otimes_{S_{|\lambda|}} \Sigma_{\lambda}
\end{equation}
(We refer to \cite{Fulton::Repr} for details on Schur-Weyl duality
and to \cite{Macdonald} ch1.8 for the detailed definition of a Schur functor).
Schur functors are defined for Young diagrams of arbitrary length.
However, if the length $l(\lambda)>\dim(U)$ then $S^{\lambda}(U)$ is zero, otherwise it is a nontrivial
irreducible $\gl(U)$-module. 
The symmetric power $S^{i}U$ is also an example of a Schur functor assigned to the Young diagram
 $\{i\}$ of length $1$. The trivial module $\kk$ is isomorphic to $S^{0}U$.

We also consider $S^{\lambda}U$ to be an irreducible $\b$-module where the action of $\gl_m$ is trivial.
\begin{lemma}
\label{lm::H_b::coeff}
The higher relative cohomology of the Lie algebra $\b$ with coefficients in 
the
$\b$-module $Hom(S^{\lambda}U;L)$
vanishes for all Young diagrams $\lambda$
of length less than or equal to $n=\dim(U)$
and for all finite-dimensional $\gl_{m+n}$-modules $L$.
The zero cohomology is isomorphic to the space of 
$\gl_{m+n}$-invariant maps between $S^{\lambda}(U\oplus V)$ and $L$:
\begin{align*}
 \rH^{>0}(\b,\gl_{m}\oplus\gl_n;Hom(S^{\lambda}(U);L)) = & 0, \\
 \rH^{0}(\b,\gl_{m}\oplus\gl_n;Hom(S^{\lambda}(U);L)) \cong & Hom_{\gl_{m+n}}(S^{\lambda}(V\oplus U); L).
\end{align*}
\end{lemma}
In particular, if $\lambda=0$ we have the following degeneration property for all $\gl_{m+n}$-modules L:
\begin{equation}
\label{eq::Hb:M}
\rH^{>0}(\b, \gl_{m}\oplus\gl_n; L) = 0 \ \text{ and } \rH^{0}(\b, \gl_{m}\oplus\gl_n; L) = [L]^{\gl_{m+n}}
\end{equation}
\begin{proof}
 The proof of this Lemma is postponed to Appendix~\ref{sec::parabolic}
where items~\ref{thm::parab::vanish::item1} and~\ref{thm::parab::vanish::item2}  
of Theorem~\ref{thm::parab::vanish} 
cover
the generalization of this statement for an arbitrary parabolic subalgebra.
Corollary~\ref{cor::Ext:gl:b} explains what one gets in the case we are interested in.
\end{proof}

\begin{lemma}
\label{lm::H_gl::to::H_b}
The relative Hochschild-Serre spectral sequence associated with an embedding $\b\hookrightarrow\gl_{m+n}$
with coefficients in any finite-dimensional $\gl_{m+n}$-module $L$ degenerates in the first term. 
More precisely, we have
$$
\rH^{p}(\bm,\gl_{m}\oplus\gl_{n};\Lambda^{q}({\np})^{*} \otimes L) =
\left\{
\begin{array}{l}
0, \text{ if } p\neq q, \\
{\begin{array}{c}
\rH^{2p}(\mathbb{C}Gr(m,m+n))\otimes [L]^{\gl_{m+n}} = \\
= \rH^{2p}(\gl_{m+n},\gl_m\oplus\gl_n;L),
\end{array} 
}
\text{ if } p=q.
\end{array}
\right.
$$
Where $\mathbb{C}Gr(m,m+n)= U_{m+n} / (U_m\times U_n)$ 
is the Grassmanian 
of
$m$-dimensional complex subspaces in 
$\mathbb{C}^{m+n}$
\end{lemma}
\begin{proof}
This lemma is also proved in Appendix~\ref{sec::parabolic} using the parabolic BGG resolution.
Part~\ref{thm::parab::vanish::item3} of Theorem~\ref{thm::parab::vanish} is a generalization of this lemma to the case 
of an arbitrary parabolic subalgebra
of a reductive Lie algebra.

The cohomology of the complex Grassmanian $\mathbb{C}Gr(m,m+n)$
is known to coincide with the cohomology 
 ring $\rH^{q}(\gl_{m+n},\gl_{m}\oplus\gl_n;\kk)$ (see e.g.~\cite{Fuks::Lie::cohomology}).
These cohomologies  are known to be even, and numbered 
by $m$-$n$-shuffle permutations (see Section~\ref{sec::parab::gl::apendix} for details).
\end{proof}

\subsection{Degeneration of Hochschild-Serre spectral sequences}
\label{sec::degen::Hoch::Serre}
In this section we prove the degenerations of Hochschild-Serre spectral sequences
corresponding to the embedding $\b\hookrightarrow W_{m+n}$ and $\b\hookrightarrow WL(m|n)$ in the first term and prove the 
surjectivity map between these spectral sequences by applying Lemmas from 
Section~\ref{sec::parabolic::gl}.
We use the same notation and description of the decomposition of Lie algebras
as in Appendix~\ref{sec::chains::gl::inv} and refer 
the reader there
 for the detailed description of chains.
Only a  sketched version is given in this section.

As 
before
we denote $V=\kk^{m}$, $U=\kk^{n}$ and $V\oplus U=\kk^{m+n}$.
We have the following decompositions as $\gl_n\oplus\gl_m=\gl(V)\oplus\gl(U)$-modules:
\begin{equation}
\label{eq::gl::W:decomp}
\begin{array}{c}
\b\cong \nm\oplus\gl_ m\oplus\gl_n \cong V^{*}\otimes U \oplus V\otimes V^{*} \oplus U\otimes U^{*},\\
{\calO_m \cong \widehat{\ooplus\limits_{i \geq 0}} S^i V^{*},}\\
{ W_{m+n} \cong \widehat{\ooplus\limits_{i \geq 0}} S^i (V\oplus U)^{*} \otimes (V\oplus U) ,} \\
{WL(m|n) \cong (V\oplus V^{*}\otimes V)\oplus \widehat{\ooplus\limits_{i \geq 0}} S^i(V\oplus U)^{*} \otimes U ,}
\end{array}
\end{equation}
where $\b:=WL(m|n)\cap \gl_{m+n}\hookrightarrow W_{m+n}$ is a parabolic subalgebra
of matrices with two diagonal blocks and lower diagonal matrices;
$\nm$ (resp. $\np$) are lower (resp. upper) block-triangular nilpotent matrices. 

The embedding $\b\hookrightarrow WL(m|n)$ produces the Hochschild-Serre filtration on
the relative chain complex $C^{\udot}(WL(m|n),\gl_{m}\oplus\gl_{n};\kk)$ and we denote by $E_{r}^{p,q}(WL(m|n))$
the corresponding spectral sequence.
Similarly, we denote by $E_r^{p,q}(W_{m+n})$ the relative Hochschild-Serre spectral sequence associated with 
the embedding $\b\hookrightarrow W_{m+n}$.
Since the Lie algebra $WL(m|n)$ is embedded into $W_{m+n}$ we 
obtain a
map of the aforementioned spectral sequences 
in the opposite direction:
$$
\pi^{r}_{m,n}: E_r^{p,q}(W_{m+n}) \longrightarrow E_{r}^{p,q}(WL(m|n))
$$
Below we prove that these spectral sequences are degenerate in the first terms and 
that the map $\pi^{1}_{m,n}$ is surjective.

\begin{lemma}
\label{lm::E1:Wm+n}
The relative Hochschild-Serre spectral sequence  $E_r^{p,q}(W_{m+n})$, 
with the embedding $\b\hookrightarrow W_{m+n}$
degenerates in the first term and
\begin{equation}
\label{eq::E1::W:b}
E_{1}^{p,q}(W_{m+n}) = \rH^{2q}(\mathbb{C}{Gr}(m,m+n))\otimes \rH^{p-q}(W_{m+n},\gl_{m+n};\kk){\cred.*}
\end{equation}
\end{lemma}
\begin{proof}
Lemma~\ref{lm::H_gl::to::H_b} implies the following description of the first term:
\begin{multline*}
E_1^{p q} = \rH^{q}\left(\b,{\gl_m\oplus\gl_n};\Lambda^{p}\left( \frac{W_{m+n}}{\b}\right)^{*}\right) =
\rH^{q}\left(\b,{\gl_m\oplus\gl_n};\Lambda^{p}\left(\np\oplus \frac{W_{m+n}}{\gl_{m+n}}\right)^{*}\right) = \\
= \rH^{q}\left(\b,{\gl_m\oplus\gl_n};\Lambda^{q}(\np)^{*}\right)\otimes 
\left[\Lambda^{p-q}\left(\frac{W_{m+n}}{\gl_{m+n}}\right)^{*}\right]^{\gl_{m+n}} =\\
= \rH^{2q}\left(\gl_{m+n},\gl_{m}\oplus\gl_n;\kk\right)\otimes 
C^{p-q}(W_{m+n},\gl_{m+n};\kk).
\end{multline*}
Theorem~\ref{thm::HWn} implies that the space of relative cochains $C^{p-q}(W_{m+n},\gl_{m+n};\kk)$
vanishes for odd $p-q$.
Therefore, $E_1^{p q}$ vanishes when $p+q$ is odd and the spectral sequence degenerates in the first term. 
\end{proof}

\begin{lemma}
\label{lm::E1:WLmn}
The relative Hochschild-Serre spectral sequences associated with the embedding $\b\hookrightarrow WL(m|n)$ 
degenerates in the first term:
$$
E_{1}^{p,>0}(WL(m|n))=0 \text{ and }  E_{1}^{p,0}(WL(m|n))= \rH^{p}(WL(m|n),\gl_{m}\oplus\gl_{n};\kk).
$$
Moreover, the morphism $\pi_{m n}^{r}: E_r^{p,q}(W_{m+n},\b)\rightarrow E_r^{p,q}(WL(m|n),\b)$ 
of the relative Hochschild-Serre spectral sequences coming from the inclusion 
of Lie algebras $\pi_{m n}:WL(m|n)\hookrightarrow W_{m+n}$ is surjective.
In particular,
$$
\rH^{\udot}(W_{m+n},\gl_{m+n};\kk) = E_{1}^{\udot,0}(W_{m+n}) \twoheadrightarrow 
E_{1}^{\udot,0}(WL(m|n))= \rH^{\udot}(WL(m|n),\gl_{m}\oplus\gl_{n};\kk).
$$
\end{lemma}
\begin{proof}
The decompositions~\eqref{eq::gl::W:decomp}  into the direct sum of finite-dimensional $\gl(U)\times\gl(V)$-modules
predicts that 
 the $\b$-module $\frac{WL(m|n)}{\b}$ 
is a direct sum of modules 
of the form $L_i\otimes U$
with each $L_i$ being an irreducible finite-dimensional $\gl(U\oplus V)$-module.
Therefore, the dual space $(\frac{WL(m|n)}{\b})^{*}$  
satisfies
  the condition from Lemma~\ref{lm::H_b::coeff}.
Let us show that the exterior algebra $\Lambda^{\udot}\left(\frac{WL(m|n)}{\b}\right)^{*}$
also decomposes into the direct sum of modules $Hom(S^{\lambda}U,L_i)$ 
so
that we can apply Lemma~\ref{lm::H_b::coeff}.
This follows from 
a 
general argument in the spirit of Howe duality~\cite{howe}:

 Indeed, let $A$ and $B$ be a pair of vector spaces. Then the exterior algebra $\Lambda^{\udot}(A\otimes B)$ has a multiplicity free
decomposition $\oplus_{\lambda} S^{\lambda^{t}}(A)\otimes S^{\lambda}(B)$ as $\gl(A)\times\gl(B)$-module,
where 
the sum is taken over all possible Young diagrams; $S^{\lambda}()$ denotes the Schur functor associated 
with the Young diagram $\lambda$ (after~\eqref{eq:Schur::Functor}). 
The Young diagram transposed to $\lambda$ is denoted by $\lambda^{t}$.

Consequently, we have the following decomposition:
\begin{multline*}
\Lambda^{\udot}\left(\frac{WL(m|n)}{\b}\right)^{*} = 
\Lambda^{\udot}\left( (V\oplus U)^{*} \ooplus\limits_{i\geq 2} S^{i}(V\oplus U)\otimes U^{*} \right)=
\\ =
\Lambda^{\udot}\left( (V\oplus U)^{*}\right)\bigotimes 
\Lambda^{\udot}\left(\ooplus\limits_{i\geq 2} S^{i}(V\oplus U)\otimes U^{*} \right)=
\\
=
\Lambda^{\udot}\left( (V\oplus U)^{*}\right)\bigotimes\left(\ooplus\limits_{\lambda} 
S^{\lambda^{t}}\left( \ooplus\limits_{i\geq 2} S^{i}(V\oplus U)\right)\otimes S^{\lambda}U^{*} \right) = \\
\bigoplus_{\lambda}\underbrace{\Lambda^{\udot}\left( (V\oplus U)^{*}\right)\bigotimes
S^{\lambda^{t}}\left( \ooplus\limits_{i\geq 2} S^{i}(V\oplus U)\right)
}_{\gl(U\oplus V)-\text{module}} \bigotimes 
\underbrace{S^{\lambda}U^{*}
}_{\gl(U)-\text{module}}
\end{multline*}
Now we are able to apply Lemma~\ref{lm::H_b::coeff} 
and
to compute the first term of
the Hochschild-Serre spectral sequence. The vanishing of the higher cohomology in Lemma~\ref{lm::H_b::coeff}
implies that for $q>0$ we have
$$E_1^{p q}(WL(m|n)) = \rH^{q}\left(\b,{\gl_m\oplus\gl_n};\Lambda^{p}\left( \frac{WL(m|n)}{\b}\right)^{*}\right) = 0.$$
For $q=0$ we have the following collection of identities and inclusion:
\begin{multline}
\label{eq::E1b::WL::Wnm}
E_1^{\udot,0}(WL(m|n)):= \rH^{0}\left(\b,{\gl_m\oplus\gl_n};\Lambda^{\udot}\left(\frac{WL(m|n)}{\b}\right)^{*}\right) = 
\\ =
\rH^{0}\left(\b,{\gl_m\oplus\gl_n};\left(
\bigoplus_{\lambda}\left(\Lambda^{\udot}\left( (V\oplus U)^{*}\right)\bigotimes
S^{\lambda^{t}}\left( \ooplus\limits_{i\geq 2} S^{i}(V\oplus U)\right)\right)
\bigotimes 
S^{\lambda}U^{*} \right) \right) \subset \\
\subset
\left[ \Lambda^{\udot}\left( (V\oplus U)^{*}\right)\bigotimes\left(\ooplus\limits_{\lambda} 
S^{\lambda^{t}}\left( \ooplus\limits_{i\geq 2} S^{i}(V\oplus U)\right)\otimes S^{\lambda}(V\oplus U)^{*}\right) \right]^{\gl_{m+n}} 
=\\
=
\left[ \Lambda^{\udot}\left( (V\oplus U)^{*}\right)\bigotimes
\Lambda^{\udot}\left( \ooplus\limits_{i\geq 2} S^{i}(V\oplus U)\otimes (V\oplus U)^{*} \right)
\right]^{\gl_{m+n}}
=\\
= \left[\Lambda^{\udot}\left(\frac{W_{m+n}}{\gl_{m+n}}\right)^{*}\right]^{\gl_{m+n}} \cong
 \rH^{\udot}(W_{n+m},\gl_{m+n};\kk).
\end{multline}
The inclusion also follows from Lemma~\ref{lm::H_b::coeff}.
Note, that this is indeed an inclusion and not an isomorphism because
if the length of the diagram $\lambda$ is greater than $n$ then $S^{\lambda}U=0$, 
however, $S^{\lambda}(U\oplus V)$ may be different from zero.
The middle identity in~\eqref{eq::E1b::WL::Wnm} follows from the Howe decomposition 
for the exterior algebra of $\frac{W_{m+n}}{\gl_{m+n}}$.
Once again, we identify relative cochains on $W_{m+n}$ and cohomology using Theorem~\ref{thm::HWn}.
 
Finally, we get that 
$E_1^{p,q}(WL(m|n))$ vanishes for $q>0$ and there is a surjection from
the zero row of the first term of the spectral sequence $E_1^{\udot,0}(W_{m+n})=\rH^{\udot}(W_{m+n},\gl_{m+n};\kk)$ to the 
the zero row of the spectral sequence $E_1^{\udot,0}(WL(m|n))=\rH^{\udot}(WL(m|n),\gl_{m}\oplus\gl_{n};\kk)$. 
\end{proof}

\subsection{Final conclusions}
\label{sec::proof::HWLmn}
 Degenerations of the Hochschild-Serre spectral sequences associated with embeddings $\b\hookrightarrow WL(m|n)$
and $\b\hookrightarrow W_{m+n}$ implies the following corollary:
\begin{corollary}
\label{thm::Wn+m->Wn}
There exists a commutative diagram of surjections:
$$
\xymatrix{
\rH^{\udot}(W_{m+n},\gl_m\oplus\gl_n;\kk) 
\ar@{->>}^{\kappa_{m n}}[rr]
\ar@{->>}_{\pi_{m n}^{*}}[rd]
& &
{
\begin{array}{c}
{\rH^{\udot}(W_{m+n},\gl_{m+n};\kk) =} \\
{ = \rH^{\udot\leq 2(m+n)}(BU_{m+n})}
\end{array}
}
\ar@{->>}^{\tilde{\pi}_{m n}^{*}}[ld]
\\
& \rH^{\udot}(WL(m|n),\gl_m\oplus\gl_n;\kk) &
}
$$
the morphism $\pi_{m n}^{*}$ is the one associated with the natural embedding $\pi_{m n}:WL(m|n)\hookrightarrow W_{m+n}$.
The morphism $\kappa_{m n}$ is the augmentation of the cohomology of Grassmanian $\rH^{\udot}(\gl_{m+n},\gl_{m}\oplus\gl_{n};\kk)$.

Moreover the surjections $\tilde{\pi}_{m n}^{*}$ are compatible for different $m$ and $n$.
That is, for any pair of tuples $m\leq m'$ and $n\leq n'$ there exists a commutative diagram of surjections:
\begin{equation}
\label{diag::Wn->WL}
\xymatrix{
\rH^{\udot}(W_{m'+n'},\gl_{m'}\oplus\gl_{n'};\kk) 
\ar@{->>}^{\kappa_{m' n'}}[r] 
\ar@{->>}[d]
&
\rH^{\udot\leq 2(m'+n')}(BU_{m'+n'})
\ar@{->>}^{\tilde{\pi}_{m' n'}^{*}}[r]
\ar@{->>}[d]
&
\rH^{\udot}(WL(m'|n'),\gl_{m'}\oplus\gl_{n'};\kk) 
\ar@{->>}^{p_{m'\to m}}[d]
\\
\rH^{\udot}(W_{m+n},\gl_{m}\oplus\gl_{n};\kk) 
\ar@{->>}^{\kappa_{m n}}[r] 
&
\rH^{\udot\leq 2(m+n)}(BU_{m+n})
\ar@{->>}^{\tilde{\pi}_{m n}^{*}}[r]
&
\rH^{\udot}(WL(m|n),\gl_{m}\oplus\gl_{n};\kk) 
}
\end{equation}
\end{corollary}
\begin{proof}
 The morphism $\pi^{*}_{mn}$ is surjective because the morphism of the first terms of
Hochschild-Serre spectral sequences $E_1^{p,q}(W_{m+n})\stackrel{\pi^{1}_{mn}}{\rightarrow} E_1^{p,q}(WL(m|n))$
is surjective.
The projection $\kappa_{m n}$ is the projection on the zero line $E_1^{\udot,0}(W_{m+n})$
and the map $\pi^{*}_{m n}$ is the map of zero lines of the first terms of spectral sequences.

All horizontal arrows in Diagram~\eqref{diag::Wn->WL} are surjective by 
the
aforementioned results.
The middle vertical arrow is surjective because the map $\rH^{\udot}(BU_{N'})\rightarrow \rH^{\udot}(BU_{N})$ 
is surjective whenever $N'>N$.
The commutativity of the Diagram implies that the map $p_{m'\to m}$ is also surjective.
\end{proof}

Finally, we can show how 
Corollary~\ref{thm::Wn+m->Wn}
imply
the main result on the cohomology of the Lie algebra $WL(m|n)$:
\begin{proof}[Proof of Theorem~\ref{thm::homol::WLmn}.]
\label{proof::main}
Theorem~\ref{thm::old::WL(n|0;g)} explains that the chain complex $C^{\udot}(WL(m|n),\gl_{m}\oplus\gl_n;\kk)$ is 
quasi-isomorphic to the truncated Weyl superalgebra 
$\W^{\udot}(W_n,\gl_n)/F^{2m+1}\W^{\udot}(W_n,\gl_n)$.
Let us fix the integer parameter $n$ and vary the parameter $m$. 
Corollary~\ref{thm::Wn+m->Wn}
implies that for different $m'\geq m$ the morphism $p_{m'\to m}$ is surjective and, therefore,
 all the components of the homology $\rH^{\udot}(WL(m|n),\gl_{m}\oplus\gl_n;\kk)$ are bounded from above by 
those of the 
homology of the inverse limit of truncated Weyl algebras.
The inverse limit of truncated Weyl algebras when $m$ goes to infinity 
is the nontruncated Weyl algebra $\W^{\udot}(W_n,\gl_n)$ whose homology coincides with
the ring of characteristic classes $\rH^{\udot}(BU_n)$.
Hence, the cohomology ring $\rH^{\udot}(WL(m|n),\gl_m\oplus\gl_n;\kk)$ is a quotient of the 
polynomial ring $\kk[\Psi_2,\ldots,\Psi_{2n}]$.
The same compatibility conditions for all possible  $m'\geq m$ implies that the map
$\tilde{\pi}_{m n}^{*}:\rH^{\udot\leq 2(m+n)}(BU_{m+n}) \rightarrow \rH^{\udot}(WL(m|n),\gl_{m}\oplus\gl_n;\kk)$
factors through the map $\rH^{\udot\leq 2(m+n)}(BU_{m+n}) \rightarrow \rH^{\udot\leq 2(m+n)}(BU_{n})$ which
sends 
 the additional generators $\Psi_{2(n+1)},\ldots,\Psi_{2(m+n)}$ to zero.
Finally, we 
end
up with the isomorphism $\rH^{\udot}(WL(m|n),\gl_{m}\oplus\gl_n;\kk)$ and the truncated polynomial ring
$\kk^{\udot\leq 2(n+m)}[\Psi_2,\ldots,\Psi_{2n}]$.
\end{proof}

The proof of Theorem~\ref{thm::HWm|nnng} follows from the same degeneration properties 
of the Hochschild-Serre spectral sequences corresponding 
to the parabolic subalgebra $\p:=WL(m|n_1,\ldots,n_k;\g)\cap \gl_{m+n_1+\ldots+n_k}$.

\section{Particular computations}
\label{sec::examples}
\subsection{dimension series for absolute cohomology of $W(1,\ldots,1)$.}
\label{sec::full_flag_formulas}
Relative cohomological classes of the Lie algebra $W(n_1,\ldots,n_k)$ 
(relative to the product of orthogonal groups $O(n_1)\times\ldots\times O(n_k)$)
 are in one-to-one correspondance 
with the continuous characteristic classes of flags of foliations.
The absolute cohomology of the Lie algebra $W(n_1,\ldots,n_k)$ corresponds 
to the case of framed foliations.
We omit the detailed construction based on the concept of formal geometry 
and refer the reader to~\cite{Feigin::flag::foliations}.
In this section we will show an example of an application of Theorem~\ref{thm::HWnm::absolute} in order to describe 
these cohomological classes.
Corollary~\ref{cor::W1..1} below describes the absolute cohomology classes of the Lie algebra $W(1,\ldots,1)$. 

\begin{remark}
 The orthogonal group $O(1)$ is trivial and therefore the normal bundle of the foliation of codimension $1$ 
is trivial. Similarly, the relative normal bundles of $\F_{i+1}\subset \F_{i}$ inside a flag of foliations 
$\{\F_1\supset\F_2\supset \dots\supset \F_k\}$ 
should be trivial whenever the corresponding codimension is equal to $1$.
Hence, for the case of full flags (that is,  when all codimensions are equal to $1$)  the orthogonal group is trivial
and the space of characteristic classes are the same for framed and nonframed foliations.
\end{remark}

Let $\zeta_i$ (respectively $\xi_i$) be the 
$i$th
odd (respectively even) generator of the Weyl superalgebra \\
$\W^{\udot}(\underbrace{\gl_1\oplus\ldots\oplus\gl_1}_{k \text{ copies}}).$
That is $\zeta_i$ (resp. $\xi_i$) is a basis of $\Lambda^{1}\gl_1^{*}$ (resp. $S^1\gl_1^{*}$) associated with the 
$i$th 
copy 
of $\gl_1$.
Let $I:= I_{1,\ldots,1}$ be the ideal of Weyl superalgebra as defined in Corollary~\ref{thm::relative_har_flag_sloj}. 
That is, $I$ is  generated by subspaces $S^{j+1}\left(\underbrace{\gl_1\oplus\ldots\oplus\gl_1}_{j \text{ copies}}\right)$ 
with $j$ ranges from $1$ to $k$.
\begin{theorem}
\label{thm::monom::W1:1}
The set of monomials 
\begin{equation}
\label{eq::monom::W11}
\left\{
\zeta_{\alpha_1}\ldots\zeta_{\alpha_s}\xi_1^{i_1}\xi_2^{i_2}\dots\xi_{\alpha_s}^{i_{\alpha_s}} 
\ : \
1\leq \alpha_1<\ldots<\alpha_s\leq N \text{ and }
\begin{array}{c}
\forall k\leq \alpha_s \quad i_1+\ldots+i_k \leq k, \\
 i_1+\ldots+i_{\alpha_s} = \alpha_s
\end{array}
\right\}
\end{equation}
form a basis of cohomological classes of the truncated Weyl superalgebra  
$\frac{\W^{\udot}(\overbrace{\gl_1\oplus\ldots\oplus\gl_1}^{N \text{ copies}})}{I_{1,\ldots,1}}$. 
The Poincare generating series of the homology of the truncated Weyl superalgebra is as follows:
\begin{equation}
\label{eq::ser::HW1..1}
\sum_{q\geq 0}q^{k}\dim\rH^{k}(\W^{\udot}(\gl_1\oplus\ldots\oplus\gl_1)/I_{1,\ldots,1})=1+\sum_{n=1}^{N} q^{2n+1}(1+q)^{n-1}C(n),
\end{equation}
where 
$\displaystyle{C(n)=\frac{1}{(n+1)}{\binom{2n}{n}}}$
 is the 
$n$th Catalan number.
\end{theorem}
\begin{proof}
Consider the second term of the spectral sequence associated with the standard filtration on the truncated Weyl superalgebra:
$$
E_2 = \kk[\zeta_1,\ldots,\zeta_N;\xi_1,\ldots\xi_N]/(I_{1,\ldots,1}), 
\quad d_2 = \sum_{i=1}^{N}\xi_i\frac{\partial}{\partial\zeta_i}.
$$
All higher differentials in this spectral sequence vanish.
Recall that the ideal $I$ is generated by the following collection of monomials:
$$
\xi_1^{i_1}\cdot \ldots\cdot \xi_k^{i_k} \text{ with } i_1+\ldots+i_k > k
$$
Let us prove that monomials yielding the restriction~\eqref{eq::monom::W11} in the statement of Theorem form a basis of 
the cohomology of the complex given by the second term $(E_2,d_2)$.
The differential $d_2$ is a direct sum of $N$ commuting differentials $\xi_i\frac{\partial}{\partial\zeta_i}$.
Let us compute the cohomology of $E_2$ with respect to the differential $\xi_N\frac{\partial}{\partial\zeta_N}$.
In other words, we consider the spectral sequence associated with the bicomplex where differential $d_2$ is the 
sum of two commuting differentials $\xi_N\frac{\partial}{\partial\zeta_N}$ 
and $\left(\sum_{i=1}^{N-1}\xi_i\frac{\partial}{\partial\zeta_i}\right)$.
The differential $\xi_N\frac{\partial}{\partial\zeta_N}$ 
maps monomials to monomials. Therefore, representing cocycles may be also chosen to be monomials.
Consider a monomial $f=\zeta_1^{\epsilon_1}\ldots\zeta_N^{\epsilon_N}\xi_1^{i_1}\ldots\xi_N^{i_N}$.
In order to be nonzero in the fraction of Weyl algebra by monomial ideal $I_{1,\ldots,1}$ we have the following restrictions:
$$
\forall i=1\ldots N \ \epsilon_i\in\{0,1\}, \ \text{ and } \forall k=1\ldots N \text{ we have } i_1+\ldots+i_k\leq k.
$$
There are two possibility for the monomial $f$
to represent a nonzero cohomological class with respect to the differential $\xi_N\frac{\partial}{\partial\zeta_N}$.
Either $\epsilon_N=i_N=0$ or $\epsilon_N=1$, $i_N\geq 1$ and $i_1+\ldots+i_N=N$.
The first case implies that we can forget about variables $\xi_N$ and $\zeta_N$ and deal with the problem of chasing cocycles
for the 
Lie algebra $W_{\underbrace{1,\ldots,1}_{N-1 \text{ copies }}}$ and use the induction arguments afterwords
In the second case the condition $i_N\geq 1$ and $i_1+\ldots+i_N=N$ implies that $\xi_i(f)$ belongs to the ideal $I$ for all $i$
and, consequently, $\xi_i\frac{\partial}{\partial\zeta_i}(f) = 0$ what means that 
monomial $f$ represents a nonzero cocycle in the full truncated Weyl algebra. 

 In order to count the Poincare series it remains to count the generating series of the set~\eqref{eq::monom::W11}
of monomials representing the linear independent homological classes.
Recall that the number of monomials $\{\xi_1^{i_1}\ldots,\xi_n^{i_n}\}$ of degree $n$ subject to the condition 
$i_1+\ldots+i_k\leq k$ for all $k\leq n$ is equal to the $n$'th Catalan number 
$C(n)=\frac{1}{(n+1)}{\binom{2n}{n}}$. This description of Catalan numbers may be found in~\cite{Stanley}.
Therefore, the generating series of monomials of the form 
$\zeta_1^{\epsilon_1}\ldots\zeta_n^{\epsilon_n}\xi_1^{i_1}\ldots\xi_n^{i_n}$,
that belong to the set~\eqref{eq::monom::W11}, is equal to $(1+q)^{n-1} q^{1}q^{2n} C(n)$ where the factor $(1+q)^{n-1}$
means that $\epsilon_i$ for $i<n$ may be either $0$ or $1$, factor $q^1$ comes from the degree of $\zeta_n$ and the factor $q^{2n}$
represents the degree of the monomial $\xi_1^{i_1}\ldots\xi_n^{i_n}$ which is $2(i_1+\ldots+i_n)=2n$.
The final set of monomials is the union of the aforementioned sets with $n$ ranging from $1$ to $N$.
\end{proof}

Consider a flag of foliations
$\F_1\supset\F_2\supset \dots\supset \F_N$, 
for which the codimensions 
$\F_{i+1}$ in $\F_i$ are equal to $1$ for all $i$.
Let $\omega_1$,\ldots,$\omega_N$ be the system of determining forms that defines this flag of foliations.
That is we have $\forall k=1$,\ldots,$N$ $d\omega_k = \sum_{i=1}^{k} \omega_i\wedge \nu_{i k}$ and 
the tangent  spaces to the leaves of the foliation $\F_k$ 
are
annihilated by 
the forms $\omega_1$,\ldots,$\omega_k$.
The space of characteristic classes of a flag of foliations are described by Theorem~\ref{thm::monom::W1:1}
in the following
\begin{corollary}
\label{cor::W1..1}
 The Poincare series of the cohomology of the Lie algebra $W(1,\ldots,1)$ with trivial coefficients is given
by 
\eqref{eq::ser::HW1..1}.
Each monomial $\zeta_i$ and $\xi_i$ from~\eqref{eq::monom::W11} defines a characteristic class of a flag of 
foliations with subsequent codimentions $1$ which is defined by a collection of $1$-forms
$\omega_k$ with $d\omega_k = \sum_{i=1}^{k} \omega_i\wedge \nu_{i k}$ subject to the substitution $\zeta_i := \nu_{ii}$ and $\xi_i:=d\nu_{ii}$.
In particular, the corresponding cohomological classes do not depend on the choice of 
the determining forms $\omega_i$.
\end{corollary}
\begin{proof}
We use Theorem~\ref{thm::W_flag_absolute} in order to identify the cohomology ring of the Lie algebra $W(1,\ldots,1)$
and the truncated Weyl superalgebra.
The generalization of 
the
Godbillon-Vey class is also straightforward. 
We refer to~\cite{Fuks::Lie::cohomology} 
for  a detailed description of the Godbillon-Vey class using the Lie algebra homology 
of formal vector fields on the line. 
\end{proof}

\subsection{Formulas for cocycles in the case of a line ($n=1$)}
\label{sec::cocycles::W1}
For the case of vector fields on a line there are some simplifications of 
the
direct description of cochains representing cocycles for the Lie algebra cohomology
$\rH^{\udot}(W_n;S^{m}W_n^{*})$.
These simplifications are based on a direct description 
of the space of chains which is simpler for the case of 
a line.
Indeed, let us identify the space of $q$-linear functionals on the Lie algebra $W_1$ with the ring of polynomials
$\kk[y_1,\ldots,y_q]$ using the following identification of basis in these two spaces.
With a monomial   
$f:=y_1^{r_1}\dots y_q^{r_q}$
we associate a $q$-linear functional $D_{f}: W_1^{\otimes q} \rightarrow \kk $ in the following way:
\begin{equation}
\label{eq::W1::chains}
D_f:
\left(\sum_{r\geq 0} a_{1 r}x^r\frac{\partial}{\partial x},\dots,
\sum_{r\geq 0} a_{q r}x^r\frac{\partial}{\partial x} \right)
\mapsto r_1!\dots r_q! a_{1 r_1}\dots a_{q r_q}.
\end{equation}
Then the space of chains $C^{p}(W_1;S^{m}W_1^{*})= Hom(\Lambda^{p}W_1\otimes S^{m} W_1;\kk)$ 
is 
identified with the subspace of polynomials in $p+m$ variables
$\kk[y_1,\ldots,y_p;z_1,\ldots,z_m]$, 
which
are skew-symmetric with respect to 
permutations of the first $p$ variables 
and symmetric with respect to 
permutations
of the last $m$ variables.
The differential $d:C^{p}(W_1;S^m W_1^*)\rightarrow C^{p+1}(W_1;S^m W_1^*)$ 
is described by the following formula (see e.g. the description given in~\cite{Fuks::Lie::cohomology}[\S2.3]  for $m=0$):
\begin{multline*}
{dP(y_1,\dots,y_{p+1};z_1,\dots,z_m) =}\\
{\sum_{1\leq s <t \leq p+1}
(-1)^{s+t-1}(y_s-y_t)P(y_s+y_t,y_1, \dots,\hat{y_{s}},\dots,\hat{y_{t}},\dots,y_{p+1};z_1,\dots,z_m)+ }\\
{+ \sum_{\substack{{1\leq s \leq p+1}\\{1\leq t\leq m}}}
 (-1)^{s}(y_s-z_t)P(y_1,\dots,\hat{y_s},\dots,y_{p+1};y_s+z_t,z_1,\dots,\hat{z_t},\dots,z_m).
}
\end{multline*}
\begin{theorem}
The cochains
\begin{gather*}
{a_{2 m}:=(y_1^{2}-y_2^{2})z_1\dots z_m\in C^2(W_1;S^m W_1^*), }\\
{a_{3 m}:=(y_1-y_2)(y_2-y_3)(y_3-y_1)z_1\dots z_m\in C^3(W_1;S^m W_1^*)}
\end{gather*}
represent a 
basis in the cohomology $\rH^{\pt}(W_1;S^{m}W_1^*)$.
\end{theorem}
\begin{proof}
A direct check shows
that $d(a_{2m})=d(a_{3m})=0$ 
{which} 
means that they are cocycles.
Moreover, $a_{2m}$ belongs to the relative chain complex $C^{2}(W_1,\gl_1;S^m W_1^*)$ and is equal 
to the cocycle $\xi_{1^{m+1},m}$ which is defined in Proposition~\ref{prop::cocycl::WnSm} 

Alternatively, one can easily verify that the natural pairing between 
the cocycle $a_{2m}$ and 
the cycle in $\rH_{2}(W_1,\gl_1;S^m W_1)$
given in~\cite{Dotsenko} is different from zero for all $m$. This will imply that the corresponding cohomological class 
is different from zero. Theorem~\ref{thm::Wn:SWn:absolute} implies that there is only one nontrivial cohomological class
in this dimension.
\end{proof}

\appendix

\section{Chains on the Lie algebra $W_n$ and $\gl$-invariants}
\label{sec::chains::gl::inv}
This section contains 
detailed explanations of 
known results and their generalizations mentioned in
Section~\ref{sec::Wn::known}.
We will describe the action of matrix Lie subalgebras on the Lie algebras $W_n$, $W(m;\g)$, $WL(m|;\g)$, $WL(m|n)$ and 
use this description in order to identify the space 
of $\gl$-invariant chains with the truncated Weyl algebras.
This description of the relative chain complexes has been written in 
detail
in~\cite{GF},~\cite{Fuks::Lie::cohomology} 
for the Lie algebra $W_n$, and \cite{Khor::Wn_g} deals with the Lie algebra $W(m;\g)$. 

\subsection{$\gl$-decompositions}
Let $V$ be a vector space of dimension $m$. Let $x_1,\ldots,x_m$ be 
a 
basis of linear coordinates on this space.
Then as a $\gl_m=\gl(V)$ module the linear span of constant derivatives $\frac{\partial}{\partial x_i}$ is isomorphic to $V$
and the space of polynomial functions of degree $k$ 
spanned by monomials $x_1^{i_1}\ldots x_m^{i_m}$ with $i_1+\ldots+i_m{\cred =k}$ 
is isomorphic to the $k$-th symmetric power of $V^{*}$. 
Thus
the space of functions $\calO_m=\calO(V)$ on $V$ is isomorphic as a $\gl(V)$-module to the 
completion of the
direct sum 
$\widehat{\ooplus_{k\geq 0}} S^{k}V^{*}$. Consequently, we have the following decompositions as $\gl_m$-modules:
$$
\begin{array}{c}
 W_m\cong \widehat{\ooplus\limits_{k\geq 0}} S^{k}V^{*}\otimes V , \qquad \gl_m\cong V^{*}\otimes V \\
W(m;\g)=W_{m}\ltimes\g\otimes \calO_m \cong_{\gl_m} 
\left(\widehat{\ooplus\limits_{k\geq 0}} S^{k}V^{*}\otimes V \right) \oplus
\widehat{\ooplus\limits_{k\geq 0}} S^{k}V^{*}\otimes \g ,
\\
WL(m|;\g)\cong_{\gl_m} 
\left(\widehat{\ooplus\limits_{k=0,1}} S^{k}V^{*}\otimes V \right) \oplus
\widehat{\ooplus\limits_{k\geq 0}} S^{k}V^{*}\otimes \g  
\cong_{\gl_m} (V^{*}\oplus \gl(V))\oplus \widehat{\ooplus\limits_{k\geq 0}} S^{k}V^{*}\otimes \g.
\end{array}
$$
Let us now describe the space of relative chains in the aforementioned Lie algebras.
First, we recall, that these infinite-dimensional Lie algebras are topological.
The underlying topology
comes 
from the topology on formal power series.
The linear continuous functions on the  vector space of formal power series 
is the polynomial ring on the dual space.
In particular, we have the following identity:
$$
\left(\widehat{\ooplus\limits_{k\geq 0}} S^{k}V\right)^{*} \cong \left({\ooplus\limits_{k\geq 0}} S^{k}V^{*}\right)
$$
We use the usual dual sign ``$*$'' having in mind that one should take the linear continuous dual.

\subsection{Graphs representing chains on $W_m$}
\label{sec::chains::Wn::graphs}
We have the following description of the relative chain complexes: 
\begin{multline}
\label{eq::Wm::chain::description}
C^{\udot}(W_m,\gl_m;\kk) = \left[\Lambda^{\udot}\left(\frac{W_m}{\gl_m}\right)^{*}\right]^{\gl_m}=
\left[\left(\widehat{\ooplus\limits_{
k=0,2,3,\ldots
}} 
S^{k}V^{*}\otimes V \right)\right]^{\gl(V)} 
\cong 
\\
\cong
\left[ \Lambda^{\udot}\left({\ooplus\limits_{k=0,2,3,\ldots}} 
S^{k}V\otimes V^{*} \right)\right]^{\gl(V)} 
\cong
\ooplus\limits_{ \{p_0,p_2,p_3,\ldots\} }
\left[\OT\limits_{k=0,2,3,\ldots} \Lambda^{p_k}(S^{k}V\otimes V^{*})\right]^{\gl(V)}.
\end{multline}
The index $k=1$ is omitted in the right-hand side of Isomorphism~\ref{eq::Wm::chain::description} 
because we factorize $W_n$ by the subspace $\gl(V)\cong V^{*}\otimes V$.
The main observation due to Gelfand and Fuchs in~\cite{GF} is the description 
 of this 
 ring of $\gl(V)$-invariants. 
They observed that all nonzero $\gl$-invariants 
come from the subring 
$$\left[\OP\limits_{p_0} \Lambda^{p_0} V^{*} \otimes \Lambda^{p_0} (S^{2}V\otimes V^{*})\right]^{\gl(V)}.$$

Let us explain their result using the language of graphs as one always 
does when
 working with $\gl$-invariants.
(see e.g.~\cite{Kontsevich::graphs} for the definition of the graph-complex in the similar problem of describing 
the cohomology of the Lie algebra of Hamiltonian vector fields.
Consider an oriented graph $\Gamma$ yielding the following conditions:

($\imath$) Each vertex has exactly one outgoing edge; 

($\imath\imath$)
For each vertex $v$ the number $in(v)$ of incoming edges is never equal to $1$. 
That is, we allow  $0,2,3$,\ldots for possible numbers of incoming edges $in(v)$.

Assign to  $\Gamma$ the relative cochain $c_{\Gamma}$ on the Lie algebra $W_n$ using the right-hand side of 
the isomorphism in~\eqref{eq::Wm::chain::description}.
Suppose that $\Gamma$ has $p_0$ vertices with no incoming edges,
$p_2$ vertices with $2$ incoming edges and so on. 
Then $\Gamma$ defines a $\gl(V)$-invariant in the space
$\OT_{k=0,2,3,\ldots} \Lambda^{p_k}( S^{k} V\otimes V^{*})$ 
where each edge in $\Gamma$ defines a $\gl(V)$-invariant pairing between appropriate factors $V^{*}$ and $V$.
To be strict one has to do the following:
first,  fix an order  on the set of vertices: $v_1$,\ldots,$v_{\sum p_i}$; 
second, attach to each vertex with $k$ incoming edges a factor $S^{k}V\otimes V^{*}$;
third, each edge $v_i\rightarrow v_j$ defines a $\gl(V)$-invariant pairing between $V^{*}$ and $V$
coming from the contravariant component of $S^{in(v_i)}V\otimes V^{*}$ and one of the covariant arguments of $S^{in(v_j)}V\otimes V^{*}$;
fourth,
make a skewsymmetrization 
with respect to  all vertices of the same 
valency. (See e.g.~\cite{Fuks::Lie::cohomology} for the detailed description.)
The latter skewsymmetrisation procedure produces 
the 
following two necessary conditions on $\Gamma$ 
in order to have a nonzero cochain $c_{\Gamma}$:
\begin{wrapfigure}{r}{4.2cm}
\includegraphics{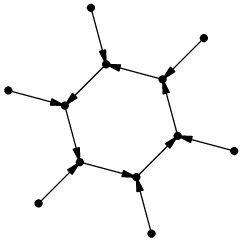}
\caption{Graph-wheel\,$\Gamma_{6}$}
\label{pic::wheel_wn}
\end{wrapfigure}

($1$) 
 If two vertices both have a zero number of incoming edges, then their outgoing arrows lead to different vertices.

($2$)
The number of vertices $v$ with  $in(v)=0$ is less or equal to the dimension of $V$.

\noindent
{Condition} ($1$) comes from the following fact:
{A} permutation of two vertices of $\Gamma$,
 each having no incoming arrows should send $c_{\Gamma}$ to $-c_{\Gamma}$, 
however the permutation of two incoming edges 
in any vertex should not change the cochain $c_{\Gamma}$. 
Therefore, whenever the first condition is not satisfied we get $c_{\Gamma}=-c_{\Gamma}$ and hence $c_{\Gamma}=0$.
Condition ($2$) corresponds to the fact that $\Lambda^{>dim(V)}(V^{*})=0$.

The conditions ($1$) and ($2$) are already enough to describe all possible graphs representing nontrivial cochains.
Let $\Gamma^{c}$ be a connected component of the graph $\Gamma$ satisfying properties ($1$) and ($2$). 
Then the set of vertices $\{v\in \Gamma^{c}| in(v)>2\}$ is empty and the cardinality of subsets of vertices 
$\{v\in \Gamma^{c}| in(v)=2\}$ and $\{v\in \Gamma^{c}| in(v)=0\}$ are the same.
Thus, the number of vertices in the connected component $\Gamma^{c}$ 
is even and there is only one connected graph $\Gamma_r$ with $2r$ vertices
yielding property ($1$) which looks like a wheel, see Picture~\picref{pic::wheel_wn} below. 
Property ($2$) implies that $c_{\Gamma}\neq 0$ only if $\Gamma=\Gamma_{r_1}\sqcup\ldots\sqcup\Gamma_{r_k}$
with $r_1+\ldots+r_k\leq dim(V)$. 
A
direct check shows that the cochains $c_{\Gamma}$ with $\Gamma$ being a union of wheels $\Gamma_{r_i}$ are linearly independent.
The map $\Psi_{2r}\rightarrow c_{\Gamma_r}$ defines an isomorphism from the truncated ring
$\rH^{\udot\leq 2n}(BU_n)=\kk^{\udot\leq 2n}[\Psi_2,\ldots,\Psi_{2n}]$
to the relative cochain complex $C^{\udot}(W_n,\gl_n;\kk)$
and we get the conclusion of Theorem~\ref{thm::HWn}.

\subsubsection{Chains on $WL(m|;\g)$.}
The relative chain complex  of the Lie algebra $WL(m|;\g)$ has 
the following description in terms of $\gl$-invariants:
\begin{multline*}
C^{\udot}(WL(m|;\g),\gl_m;\kk)\cong 
\left[\Lambda^{\udot}\left((V^{*}\oplus \gl(V))\oplus \widehat{\ooplus\limits_{k\geq 0}} S^{k}V^{*}\otimes \g\right)^{*}\right]
=
\\
=
\bigoplus_{ \{q_0;p_0,p_1,\ldots\} }
\left[\Lambda^{q_0}V^{*}\OT\bigotimes_{k\geq 0} \Lambda^{p_k}\left(S^{k}V\otimes \g^{*}\right)
\right]^{\gl(V)}
\end{multline*}
One has a description of the invariants in terms of graphs similar to the one considered for the Lie algebra $W_n$.
The main difference consists of removing vertices with $in(v)\geq 2$ and 
adding vertices with no outgoing edges and with arbitrary amount of incoming edges.
However, Condition ($1$) and ($2$) for these graphs remains the same.
The symmetry arguments show that 
there are nontrivial invariants only in the case $p_k=0$ for $k>1$ and $p_1=q_0 \leq dim(V)$. 
So we end up with an isomorphism with the truncated Weyl algebra
\begin{multline*}
C^{\udot}(WL(m|;\g),\gl_m;\kk)\cong 
\bigoplus_{
\begin{smallmatrix}
p\geq 0,
1\leq q\leq m 
\end{smallmatrix}
}
\left[\Lambda^{q}V\otimes \Lambda^{p}(\g^{*})\otimes\Lambda^{q}(V\otimes\g^{*}) \right]^{\gl(V)} 
\cong\\
\cong
\Lambda^{\udot}(\g^{*})\otimes S^{\udot\leq dim(V)}(\g^{*}) \cong \W^{\udot}(\g)/F^{2m+1}
\end{multline*}

The description of the relative chains of the Lie algebra $W(m;g)$ is somehow the union of the descriptions of 
the chains on $W_m$ and chains on $WL(m|;\g)$. We omit the details and refer 
the reader to the  
\begin{wrapfigure}{r}{4.2cm}
\includegraphics{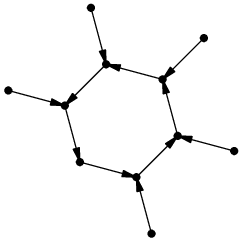}
\caption{Graph representing $\xi_{6,1}$}
\label{pic::wheel_wnSwn}
\end{wrapfigure} 
article~\cite{Khor::Wn_g}.

\subsection{formulas for cocycles}
\label{sec::cocycles}
\label{sec::cocycles::general}
All  the previous sections of Appendix~\ref{sec::chains::gl::inv} 
were collected 
here in order to be able to give a description 
of {the} cocycles representing the cohomology classes of $\rH^{\udot}(W_n,\gl_n;S^{m}W_n^{*})$. 
The main idea is to describe the images of cocycles coming from the cohomology {classes} of $\rH^{\udot}(W_{n+m},\gl_{n+m};\kk)$. 

In order to distinguish 
between the present case and
the previous one
we denote the canonical $n$-dimensional vector space $\kk^{n}$ by $U$
and the $m$-dimensional vector space is denoted by $V$.
Consider the $\gl_{m+n}$-invariant 
$$\xi_r\in\Lambda^{r}(V^{*}\oplus U^{*})\otimes \Lambda^{r}\left(S^{2}(V\oplus U)\otimes (V^{*}\oplus U^{*})\right)$$
represented by the wheel graph $\Gamma_r$.
Let $\varphi(\xi_r)$ be the restriction of this invariant 
to 
the subspace
$$\Lambda^{r}(V^{*}\oplus U^{*})\otimes \Lambda^{r}(S^{2}(V\oplus U)\otimes U^{*})$$
The latter subspace is decomposed into the following direct sum:
\begin{multline}
\Lambda^{r}(V^{*}\oplus U^{*})\otimes \Lambda^{r}(S^{2}(V\oplus U)\otimes U^{*})  \cong \\
\cong \bigoplus_{
\begin{smallmatrix}
s_1+s_2 = r\\
t_1+t_2+t_3 =r
\end{smallmatrix}
}
 \Lambda^{s_1}V^{*}\otimes \Lambda^{s_2} U^{*} 
\otimes \Lambda^{t_1}(S^{2}V\otimes U^{*}) \otimes \Lambda^{t_2}( V\otimes U\otimes U^{*}) 
\otimes \Lambda^{t_3}( S^{2}U\otimes U^{*})
\end{multline}
and the symmetry conditions implies that 
$\gl(V)$-invariants are nontrivial only
 for $t_1=0$, $s_1=t_1$ and $s_2=t_2$ respectively. 
Moreover, we have the following description of the subspace of $\gl(V)$-invariants:
\begin{multline*}
\left[\Lambda^{s}V^{*}\otimes \Lambda^{r-s} U^{*}\otimes \Lambda^{s}( V\otimes U\otimes U^{*})
\otimes \Lambda^{r-s}(S^{2}U\otimes U^{*})\right]^{\gl(V)} \cong \\
\cong
\Lambda^{r-s} U^{*}\otimes \Lambda^{r-s}(S^{2}U\otimes U^{*}) 
\otimes\left[\Lambda^{s}V^{*} \otimes \Lambda^{s}( V\otimes U\otimes U^{*}) 
\right]^{\gl(V)} \cong\\
\cong
\Lambda^{r-s} U^{*}\otimes \Lambda^{r-s}(S^{2}U\otimes U^{*}) 
\otimes S^{s}(U\otimes U^{*}).
\end{multline*}
We denote by $\xi_{r,s}$ the image of the projection $tr_s$ of $\varphi(\xi_r)$ onto the aforementioned subspace of $\gl(V)$ invariants:
$$
tr_s:\left[\Lambda^{r}(V^{*}\oplus U^{*})\otimes \Lambda^{r}(S^{2}(V\oplus U)\otimes U^{*})\right]^{\gl(V)}
\twoheadrightarrow
\Lambda^{r}U^{*}\otimes \Lambda^{r}(S^{2} U\otimes U^{*})\otimes S^{s}(U\otimes U^{*})
$$
where 
{the} parameter $s$ may 
{range} from $1$ to $r$.
The  underlying graphs associated with the invariant $\xi_{r,s}$  also look like a union of wheels. 
The main difference with the wheel graphs $\Gamma_r$ is that now some vertices on the ring of 
a wheel will have no incoming edges from the outside of the ring. (See Picture~\picref{pic::wheel_wnSwn}).

\begin{proposition}
\label{prop::cocycl::WnSm}
For any partition $\lambda=\{\lambda_1\geq\ldots\geq\lambda_k\}$ of the integer $(m+n)$ the cochain 
$$
\xi_{\lambda,n}:= \sum_{s_1+\ldots+s_k=m} \xi_{\lambda_1,s_1}\cdot\ldots\cdot\xi_{\lambda_k,s_k} \in 
\left[\Lambda^{n}U^{*}\otimes \Lambda^{n}(S^{2} U\otimes U^{*})\otimes S^{m}(U\otimes U^{*})\right]^{\gl(U)}
$$
defines a cocycle in $C^{2n}(W_n,\gl_n;S^{m}W_n^{*})$.
Moreover, the cocycles $\xi_{\lambda,n}$ with $l(\lambda)\leq n$ are linearly independent and form a set
of representatives of all 
{the} different cohomological classes. 
\end{proposition}
\begin{proof}
The proof is a straightforward description of the images of cocycles under surjections:
\begin{equation}
\label{eq::Wm+n::Wmn}
C^{\udot}(W_{m+n},\gl_{m+n};\kk)=\rH^{\udot}(W_{m+n},\gl_{m+n};\kk)\twoheadrightarrow 
\rH^{\udot}(WL(m|n),\gl_m\oplus\gl_n;\kk)\twoheadrightarrow 
\rH^{\udot}(W_n,\gl_n;S^{k}W_n^{*}).
\end{equation} 
Because of the first equality in~\eqref{eq::Wm+n::Wmn} 
the linear independence of cochains $\{\xi_{\lambda,n} | l(\lambda)\leq n\}$ will imply the linear independance of
corresponding cohomological classes.
There exists a natural way how to assign a chain $\eta_{\Gamma}\in C_{\ldot}(W_n,\gl_n;S^{k}W_n^{*})$ to a graph $\Gamma$
(see e.g.~\cite{Fuks::Lie::cohomology}).
The pairing between  $\xi_{\lambda,n}$ and $\eta_{\Gamma_{\mu}}$ with $l(\lambda),l(\mu) \leq n$ is nondegenerate and
thus the cochains $\xi_{\lambda,n}$ are linearly independent.
\end{proof}

\section{Relative cohomology of parabolic Lie subalgebra}
\label{sec::parabolic}
In this section we will explain some vanishing cohomological results for the cohomology of a parabolic subalgebra.
All these results are simple applications of the BGG resolution.
We formulate all key corollaries in their full generality.
That is, we work with a semisimple or reductive Lie algebra $\g$ and 
negative parabolic subalgebra $\p_{I}$.
However, the applications we need in order to compute the Lie algebra cohomology of 
{a} certain class of vector fields
is a very specific case: $\g=\gl_{m+n}$ and 
{the} Levi subalgebra of 
{the} parabolic subalgebra $\p_{I}$ is the subset of matrices 
with two blocks $\h_{I}=\gl_m\oplus\gl_n$.

\subsection{Notation}
\label{sec::parabolic::notation}
Let $\g$ be a semisimple (reductive) Lie algebra with a chosen Cartan decomposition
$\g = \n^{-}\oplus\h\oplus \n^{+}$.  
Let $R$ be the associated root system, $S$ the subset of simple roots and $W$ the corresponding Weyl group.
Each element $w\in W$ admits a decomposition $w=s_{\alpha_1}\ldots s_{\alpha_r}$ into the product of simple reflections $s_{\alpha}$ with $\alpha\in S$. The minimum of the number of factors among such decompositions is called the length of $w$ and is denoted $l(w)$. 
The subalgebra $\p$ is called parabolic if it contains the negative Borel subalgebra ${\mathbf{b}^{-}}=\h\oplus\n^{-}$.
Let us fix a subset $I$ of simple roots and assign to it 
the
standard parabolic subalgebra $\p_{I}$.
Let $R_I$ be the subset of roots generated by $I$. 
 $R_I^{+}$ and $R_I^{-}$ 
{denote}
the subsets of positive and negative roots of $R_I$.

Now we are able to specify the notations for 
{some important} 
subalgebras related to a chosen parabolic subalgebra $\p_{I}$: 
$$
\begin{array}{rl}
\p_{I} := & \n^{-}\oplus\h\oplus\bigoplus_{\alpha\in R_{I}^{+}}\g_{\alpha} \text{ (Standard parabolic subalgebra)}, \\
\h_{I} := & \h \oplus \bigoplus_{\alpha\in R_{I}}\g_{\alpha} \text{ (Maximal reductive subalgebra of } \p_{I}\text{)},\\
\n_{I}^{+}:= & \bigoplus_{\alpha\in R^{+}\setminus R_{I}^{+}}\g_{\alpha} \\
\n_{I}^{-}:= & \bigoplus_{\alpha\in R^{-}\setminus R_{I}^{-}}\g_{\alpha} \text{ (Nilradical of }  \p_{I}\text{)}.
\end{array}
$$
{ Alternatively, we have the} decomposition $\g=\n_{I}^{-}\oplus \h_{I}\oplus \n_{I}^{+} = \p_{I}\oplus \n_{I}^{+}$ 
which we 
 call the generalized Cartan decomposition. 

In general, for any parabolic subalgebra $\p$ it is possible 
to choose a Borel subalgebra and a set of generators $I$ such that $\p=\p_{I}$. 
(See e.g.~\cite{Fulton::Repr,Humphreys::Repr} for details on parabolic Lie subalgebras.)
Moreover, it is known that the corresponding description of {the} BGG category $\mathcal{O}$ is absolutely parallel 
to the standard one (see e.g.~\cite{Humphreys::BGG} for details on BGG).
In particular, the parabolic BGG resolution was introduced in~\cite{Lepovsky} 
and may be 
{briefly} summarized in the following way:

Let $P^{+}\subset P$ be the set of integral dominant weights for the root system $R$ and Cartan subalgebra $\h$.
Let $P_{I}^{+}\subset P_{I}=P$ be the set of integral dominant weights for the root system $R_{I}$ and same Cartan subalgebra $\h$.
{That is} the elements of $P^{+}$ are in one-to-one correspondence with irreducible finite-dimensional $\g$-modules and
the elements of $P_{I}^{+}$ correspond to irreducible finite-dimensional $\h_{I}$-modules.
By $W_I\subset W$ we denote the Weyl group associated with $R_I$ and by $W^{I}$ we denote the right coset $W_I\backslash W$
consisting of elements of 
minimal length in $W$.

Any weight $\lambda\in P_{I}^{+}$ defines a finite-dimensional irreducible $\h_{I}$-module $L_{I}(\lambda)$
 with 
 highest weight $\lambda$ 
{as well as} the parabolic Verma module $M_{I}(\lambda):=U(\g)\otimes_{U(\p_{I})}L_{I}(\lambda)$.

\begin{theorem*}[\cite{Lepovsky}]
Let $\lambda\in P^{+}$ be a regular dominant weight, 
{and suppose that} 
$L(\lambda)$ be the corresponding irreducible $\g$-module 
with highest weight $\lambda$.
Then there is an exact sequence 
\begin{equation}
\label{eq::BGG::parabolic}
\ldots \to \ooplus_{\omega\in W^{I}, l(\omega)=k} M_{I}(\omega\cdot \lambda)
\to \ldots \to M_{I}(\lambda) \twoheadrightarrow L(\lambda) \to 0
\end{equation}
\end{theorem*}
For a nonregular integral dominant weight $\lambda\in P^{+}$ a similar resolution exists. 
The $k$-th 
{term} is isomorphic to the direct sum of 
all 
$M_{I}(\omega\cdot \lambda)$
such that $\omega$ is an element of the minimal length which sends $\lambda$ to $\omega\cdot\lambda$.

\subsection{Applications of BGG resolution}
\label{sec::parabolic::BGG}
We are interested in the applications of BGG for 
cohomological computations over parabolic subalgebras.
Indeed, we are going to prove the following:

\begin{theorem}
\label{thm::parab::vanish}
Let $\p_{I}\subset \g$ be a parabolic subalgebra in $\g$, $\h_{I}$ it's Levi subalgebra as above.
Then for any finite-dimensional $\g$-module $L$ we have the following
\begin{enumerate}
 \item 
\label{thm::parab::vanish::item1}
the higher relative cohomology of the parabolic subalgebra $\p_{I}$
with coefficients in $L$ vanishes and 
{elements of} the 
 {degree zero} cohomology are {all} $\g$-invariants:
\begin{equation}
\label{eq::H::p_I::g_mod}
\rH^{>0}(\p_{I},\h_{I};L) = 0 \text{ and } \rH^{0}(\p_{I},\h_{I};L) = [L]^{\g}.
\end{equation}
\item
\label{thm::parab::vanish::item3}
For any pair $\lambda\in P^{+}$ and $\mu\in P_{I}^{+}$ of integral weights, the relative 
extension groups between 
{the} irreducible $\g$-module $L(\lambda)$ and 
{the} irreducible $\p_{I}$-module $L_{I}(\lambda)$
{vanish} if $\lambda$ and $\mu$ are in 
different linkage  classes. 
Moreover, if there exists 
a
$\omega\in W^{I}$
such that $\omega\cdot \lambda = \mu$ then the cohomology 
$\rH^{\udot}(\p_{I},\h_{I};Hom(L(\lambda),L_{I}(\mu)))$ is one dimensional and has cohomological 
degree $l(\omega)$: 
$$
\rH^{\udot}(\p_{I},\h_{I};Hom(L(\lambda),L_{I}(\mu))) =
\left\{
\begin{array}{l}
 \kk[-l(\omega)], \text{ if }\omega\cdot \lambda = \mu \text{ for }\omega\in W^{I}, \\
0, \text{ otherwise. }
\end{array}
\right.
$$
\item
\label{thm::parab::vanish::item2}
The relative Hochschild-Serre spectral sequence associated with the embedding $\p_{I}\hookrightarrow\g$
{and} with coefficients in $L$ degenerates in the first term. In other words, we have the following collection of isomorphisms:
\begin{equation}
\rH^{\udot}(\g,\h_{I};L) \cong
\rH^{\udot}(\p_{I},\h_{I};L\otimes\Lambda^{\pt}({\n_{I}^{-}})) 
\cong \rH^{\udot}(\g,\h_{I};\kk)\otimes [L]^{\g} \cong
\left(\ooplus_{w\in W^{I}}\kk[-2l(w)]\right)\otimes [L]^{\g}
\end{equation}
where $W^{I}$ is the right coset $W_I\backslash W$
represented by elements of the minimal length in the Weyl group $W$ and by $\kk[-2l(w)]$ we mean the one-dimensional vector space
shifted 
by the 
homological degree $2l(w)$.
\end{enumerate}
\end{theorem}
\begin{proof}
Consider the category $(\p_{I},\h_{I})\ttt mod$ of $\p_{I}$-modules that are semi-simple as $\h_{I}$-modules.
{In other words} each module is a sum (probably infinite) of finite-dimensional $\h_{I}$-modules.
Then the
Lie algebra cohomology which we are looking for may be considered as the derived $Hom$-functor in this category.
In particular, we can use 
the 
BGG resolution $B_{\lambda}^{\udot}$  
of the irreducible $\g$-module $L(\lambda)$ as defined in~\eqref{eq::BGG::parabolic}
in order to compute the derived homomorphisms 
between $L(\lambda)$ 
and arbitrary module $N\in mod\ttt(\p_{I},\h_{I})$:
\begin{multline}
\label{eq::Hom::BGG}
\rH^{\udot}(\p_{I},\h_{I};Hom(L(\lambda),N)) = 
\rH^{\udot}(\p_{I},\h_{I};Hom(\ooplus_{\omega\in W^{I}} M_{I}(\omega\cdot \lambda)[-l(\omega)],N)) = \\
\rH^{\udot}\left(\ooplus_{\omega\in W^{I}}Hom_{\h_{I}}(L_{I}(\omega\cdot \lambda),N)[-l(\omega)]\right)
\end{multline}
In particular, if $N$ is a trivial $\g$-module then $Hom_{\h_{I}}(L_{I}(\omega\cdot\lambda),\kk)$ differs from zero only for $\omega=Id$
and $\lambda=0$. 
What means that we have checked Identity~\eqref{eq::H::p_I::g_mod} for irreducible $\g$-module $N$.
The semisimplicity of the category of finite-dimensional $\g$-modules implies that Identity~\ref{eq::H::p_I::g_mod} 
is true for all finite-dimensional $\g$-modules and item~\ref{thm::parab::vanish::item1} is 
proved. 

If $N$ is isomorphic to an irreducible $\h_{I}$-module $L_{I}(\mu)$ with 
highest weight $\mu$ then there exists at most one 
$\omega\in W^{I}$ such that $\omega\cdot\lambda = \mu$. 
Consequently, the group $Ext^{\udot}_{(\p_{I},\h_{I})\ttt mod}(L(\lambda),L_{I}(\mu))$
is at most one-dimensional and item~\ref{thm::parab::vanish::item3} is 
proved.

We will prove first the item~\ref{thm::parab::vanish::item2} for the most degenerate case, for which 
 $\p_{I}$ coincides with 
{the} Borel subalgebra $\mathbf{b}_{-}$.
Recall  the semisimple decomposition $\n^{-}\cong\ooplus_{\alpha\in R^{-}} \g_{\alpha}$ 
of $\n^{-}$ as $\h$-modules.
Each $\g_{\alpha}$-is one-dimensional and we get the $\h$-decomposition of the exterior algebra:
$$
\Lambda^{\udot}\n^{-} \cong \ooplus_{S\subset R^{-}} \Lambda_{\alpha\in S} \g_{\alpha}.
$$
In particular, $\Lambda^{top}\n^{-}$ has 
weight $-2\rho$ (the sum of all negative roots).
Therefore, the set $X$ of $\h$-eigen values in $\Lambda^{\udot}n^{-}$
is the intersection of the convex hull of the orbit $W\cdot 0$ for the dot-action 
and the root lattice. 
In particular, for any integral dominant weight $\lambda\in P^{+}$, 
the intersection of the orbit $W\cdot\lambda$ 
and $X$ is nontrivial only if $\lambda=0$.
In order to capture the possible values of $\omega$ we have to recall another definition of the length.
(See e.g.~\cite{Bourbaki}.)
The length $l(\omega)$ of an element $\omega$ of the Weyl group
may be defined as the cardinality of the set $\omega(R^{+})\cap R^{-}$.
{In other words} $\omega$ sends exactly $l(\omega)$ positive roots to negative roots.
Denote by $n_{\omega}^{-}:=\ooplus_{\alpha\in(\omega(R^{+})\cap R^{-})} \g_{\alpha}$
and $n_{\omega}^{+}:=\ooplus_{\alpha\in(R^{+}\cap \omega(R^{-}))} \g_{\alpha}$ the corresponding dual 
subspaces in $n^{-}$ and $n^{+}$ respectively.
For any element $\omega\in W$ we get the isomorphisms: 
$$
Hom_{\h}(\kk_{\omega\cdot 0},\Lambda^{\udot}(\n^{-})) = 
Hom_{\h}(\kk_{\omega\cdot 0},\Lambda^{l(\omega)}(\n_{\omega}^{-})) \cong \kk,
$$
where $\kk_{\omega\cdot 0}$ is the one-dimensional $\h$-module of weight $\omega\cdot 0$.
The semisimplicity of the category of finite-dimensional $\g$-modules implies item~\ref{thm::parab::vanish::item2} 
 for the case of {a} general $\g$-module $L$ and 
 $\p_{I}$ being a Borel subalgebra ($I$ is empty):
\begin{multline*}
\rH^{\udot}(\mathbf{b}_{-},\h;Hom(L,\Lambda^{\udot}\n^{-})) = 
\rH^{\udot}(\mathbf{b}_{-},\h;Hom(\ooplus_{\omega\in W} M(\omega\cdot 0)[-l(\omega)],\Lambda^{\udot}\n^{-}))\otimes [L]^{\g} = \\
= \rH^{\udot}\left(\ooplus_{\omega\in W}Hom_{\h}(\kk_{\omega\cdot 0},\Lambda^{\udot}\n^{-})[-l(\omega)]\right)\otimes [L]^{\g} = \\
=
\rH^{\udot}\left(\ooplus_{\omega\in W}Hom_{\h}(\kk_{\omega\cdot 0},\Lambda^{l(\omega)}\n_{\omega}^{-}[-l(\omega)])[-l(\omega)]\right)\otimes [L]^{\g} =
\left(\ooplus_{\omega\in W} \kk[-2l(\omega)]\right)\otimes [L]^{\g}.
\end{multline*}

The same arguments works in the case of {an} arbitrary parabolic subalgebra $\p_{I}$.
It is enough to mention that possible $\h$-weights in the exterior algebra $\Lambda^{\udot}\n_I^{-}$ are bounded from above by 
the weights of $\Lambda^{\udot}\n^{-}$ and, therefore, the derived hom from any nontrivial $\g$-representation $L(\lambda)$ 
to $\Lambda^{\udot}\n_I^{-}$ is zero.
Moreover, the representatives in the right coset $W_{I}\backslash W$ were chosen to have 
minimal length.
In particular, this follows that $\n_{\omega}^{-}$ belongs to $\n_{I}^{-}$ iff $\omega\in W^{I}$ 
and, therefore,
there is exactly one nontrivial $\h_{I}$-homomorphism from $L_{I}(\omega\cdot 0)$ to $\Lambda^{\udot}n_{I}^{-}$
which factors through the subspace $\Lambda^{l(\omega)}n_{\omega}^{-}$.
\end{proof}

\subsubsection{Particular case of matrices with two blocks}
\label{sec::parab::gl::apendix}
Let us explain what we get 
for 
$\g=\gl_{m+n}$ and $\p_{I}=\b$ which is 
our main 
{case of} interest.
In this case $\h$ consists of diagonal matrices and has dimension $m+n$;
$\h_{I}$ is isomorphic to $\gl_m\oplus\gl_n$.
Integral dominant weights are numbered by non-increasing sequences of integer numbers:
$$
P^{+} := \{\lambda=(\lambda_1\geq\ldots\geq \lambda_{m+n}) : \lambda_i\in\Z\} 
\text{ and }
P^{+}_{I}:= \{\lambda=(\lambda_1\geq\ldots\geq \lambda_{m};\lambda_{m+1}\geq\ldots\geq\lambda_{m+n}) : \lambda_i\in\Z\} 
$$
The full Weyl group $W$ is the symmetric group $S_{m+n}$ and the subgroup $W_{I}$ is a product $S_m\times S_n$ of two 
symmetric groups. 
Consequently, the right coset $W^{I}:=W_{I}\backslash W= S_m\times S_n\backslash S_{m+n}$ consists of $m$-$n$--shuffle permutations.
{That is,} $\omega\in S_{m+n}$ belongs to  $W^{I}$ iff $\omega(1)<$\ldots$<\omega(m)$ and $\omega(m+1)<$\ldots$<\omega(m+n)$

The dot-action $\omega\cdot\lambda$ is the standard action shifted by $-\rho$.
{For} $\g=\gl_{m+n}$ the half-sum of positive roots $\rho$ is equal to 
$\frac{1}{2}(m+n-1,m+n-3,\ldots,3-m-n,1-m-n)$. Consequently, for $\omega\in S_{m+n}$ the dot-action looks as follows:
$$
\omega\cdot \lambda = \omega(\lambda+\rho)-\rho = (\lambda_{\omega(1)}+1-\omega(1),\ldots,\lambda_{\omega(n)}+n-\omega(n))
$$

\begin{corollary}
\label{cor::Ext:gl:b}
Let $L_{I}({\mu})$ be an irreducible polynomial $\gl_n$-module with 
highest weight $\mu=(\mu_1\geq\ldots\geq\mu_n\geq 0)$.
Consider $L_{I}({\mu})^{*}=L_{I}(-\mu)$ as an irreducible $\b$-module,
so that
and of the nilpotent part $\n_{I}^{-}$ 
are 
trivial.
Then for any irreducible $\gl_{m+n}$-module $L(\lambda)$ with 
highest weight $\lambda=(\lambda_1\geq\ldots\geq \lambda_{m+n})$
relative cohomology $\rH^{\udot}(\b,\gl_m\oplus\gl_n;Hom(L(\lambda),L_{I}(-\mu)))$ vanishes except {in} the case 
$\lambda= -\mu$ where they are one dimensional 
in 
the homological degree $0$:
\begin{gather*}
 \rH^{>0}(\b,\gl_m\oplus\gl_n;Hom(L(\lambda),L_{I}(-\mu))) = 0 \ \text{ for all }\ \lambda,\mu \\ 
\text{and }\ \rH^{0}(\b,\gl_m\oplus\gl_n;Hom(L(\lambda),L_{I}(-\mu))) = \kk\ \text{ iff } \
\lambda_1=\ldots=\lambda_m=0 \ \& \ \lambda_{m+i}=-\mu_{n-i}.
\end{gather*}
\end{corollary}
\begin{proof}
We apply item~\ref{thm::parab::vanish::item3} of Theorem~\ref{thm::parab::vanish} for the case $\g=\gl_{m+n}$,
$\p_{I}=\b$ and $\h_{I}=\gl_m\oplus\gl_n$.

The weight $(0,\ldots,0\geq-\mu_{n}\geq\ldots\geq -\mu_{1})$ is dominant for the Lie algebra $\g$.
Hence, the dot-action $\omega\cdot (-\mu)$ for a nontrivial shuffle permutation $\omega\neq Id$ will 
{no longer be} a dominant weight
because in each dot-orbit of the Weyl group there is no more than one integral dominant weight.
Therefore, $\omega\cdot \lambda = -\mu$ 
if and only if 
$\omega=Id$ and $\lambda= -\mu$. 
\end{proof}

Let us reformulate Corollary~\ref{cor::Ext:gl:b} in terms of Schur functors as 
{this is needed} in Lemma~\ref{lm::H_b::coeff}.
Indeed, if $\mu$ is 
the 
highest weight of a polynomial $\gl_n$-module, that is $\mu$ is a Young diagram,
then $L_{I}(\mu)=S^{\mu}(U)$. 
Let $L(\mu):=S^{\mu}(V\oplus U)$ be the corresponding irreducible polynomial $\gl_{m+n}$-module
with the same highest weight.
Thus Corollary~\ref{cor::Ext:gl:b} implies coincidence of dimensions:
$$
\dim \rH^{0}(\b,\gl_m\oplus\gl_n;Hom(S^{\mu}U; L(\lambda))) = \delta_{\lambda,\mu} = 
\dim Hom_{\gl_{m+n}}(S^{\mu}(U\oplus V); L(\lambda)).
$$
Since the category of finite-dimensional $\gl_{m+n}$-modules is semi-simple we get an isomorphism 
for arbitrary $\gl_{m+n}$-module $L$:
\begin{equation}
\label{eq::Hb=inv}
\rH^{0}(\b,\gl_m\oplus\gl_n;Hom(S^{\mu}U; L)) \cong  Hom_{\gl_{m+n}}(S^{\mu}(U\oplus V); L).
\end{equation}
This is the reformulation of Corollary~\ref{cor::Ext:gl:b} given in Lemma~\ref{lm::H_b::coeff}.

\bigskip
\noindent
{\footnotesize 
\textbf{Anton~Khoroshkin}:\newline
Faculty of Mathematics for Economics, Vavilova 7, National Research University
Higher School of Economics, Moscow 101990, Russia
\newline
and 
ITEP, Bolshaya Cheremushkinskaya 25, 117259, Moscow, Russia
}
\\
{\it email address:} {\footnotesize $\langle$\textbf{akhoroshkin@hse.ru}$\rangle$}

\end{document}